\documentclass{amsart}

\def\clap#1{\hbox to 0pt{\hss#1\hss}}

\usepackage[english]{babel}
\usepackage{amsthm}
\usepackage{amsmath}
\usepackage{amssymb}
\usepackage{url}
\usepackage[a4paper]{geometry}
\usepackage[all]{xy}
\usepackage{flafter}

\usepackage{hyperref}

\numberwithin{equation}{section}

\theoremstyle{plain}
\newtheorem*{mainthm}{Main Theorem}
\newtheorem{theorem}[equation]{Theorem}
\newtheorem{proposition}[equation]{Proposition}
\newtheorem{lemma}[equation]{Lemma}
\theoremstyle{definition}
\newtheorem{algorithm}[equation]{Algorithm}
\newtheorem{corollary}[equation]{Corollary}
\newtheorem{definition}[equation]{Definition}
\newtheorem{remark}[equation]{Remark}

\newcommand{\RR}{\mathbf{R}}
\newcommand{\QQ}{\mathbf{Q}}
\newcommand{\Q}{\mathbf{Q}}
\newcommand{\CC}{\mathbf{C}}
\newcommand{\C}{\mathbf{C}}
\newcommand{\ZZ}{\mathbf{Z}}
\newcommand{\Z}{\mathbf{Z}}
\newcommand{\FF}{\mathbf{F}}

\newcommand{\OK}{\mathcal{O}_K}
\newcommand{\OKO}{\mathcal{O}_{K_0}}

\newcommand{\HH}{\mathcal{H}}
\newcommand{\A}{\mathfrak{a}}
\newcommand{\B}{\mathfrak{b}}

\newcommand{\krootbig}{\QQ(\sqrt{\Delta_0},
    \sqrt{-a+b\sqrt{\vphantom{a}\smash{\Delta_0}}})}

\newcommand{\krootsmall}{\QQ(\sqrt{-a+b\sqrt{\vphantom{a}\smash{\Delta_0}}})}

\newcommand{\abs}[1]{\left\lvert #1 \right\rvert}
\newcommand{\Aut}{\operatorname{Aut}}
\newcommand{\calF}{\mathcal{F}}
\newcommand{\covol}{\mathop{\mathrm{covol}}}
\newcommand{\diag}{\mathop{\mathrm{diag}}}
\newcommand{\End}{\operatorname{End}}
\newcommand{\homigu}[2]{{I_{#1}^{\vphantom{-1}#2}}}
\renewcommand{\Im}{\mathop{\mathrm{Im}}}
\newcommand{\otilde}[2]{\smash{\widetilde{O}(\Delta_1^{\smash{#1}}\Delta_0^{\smash{#2}})}}
\renewcommand{\Re}{\mathop{\mathrm{Re}}}
\newcommand{\tbt}[1]{\left(\begin{array}{cc}#1\end{array}\right)}
\newcommand{\transpose}{^{\mathrm{t}}}
\DeclareMathOperator{\Tr}{Tr}

\author{Marco Streng}
\title{Computing Igusa class polynomials}

\address{Mathematics Institute, University of Warwick, Coventry CV4 7AL, United Kingdom}
\email{marco.streng@gmail.com}
\thanks{The results in this paper were part of my PhD thesis at Universiteit Leiden, and I would like to express my gratitude to my advisor, Peter Stevenhagen. I am grateful also to Eyal Goren and Kristin Lauter for
providing me with preliminary versions of
their denominator bounds, which enabled me to start the work on
this project without waiting for their publication.
Thanks also to Dan Bernstein for providing the reference
for Section~\ref{ssec:polyfromroots}.
Finally, thanks to David Freeman, members of my PhD committee,
and the anonymous referee
for suggesting various improvements to the exposition.
Partially supported by EPSRC grant number EP/G004870/1.
}


\begin{document}
\begin{abstract}
We bound the running time
of an algorithm that computes the genus-two
class polynomials of a primitive quartic CM-field~$K$.
This is in fact the first running time bound and even the first proof of correctness
of any algorithm that computes these polynomials.

Essential to bounding the running time is our bound on the height
of the polynomials, which is a combination of
denominator bounds of Goren and Lauter~\cite{gorenlautertoappear}
and our own absolute value bounds.
The absolute value bounds are obtained by combining
Dupont's~\cite{dupont}
estimates of theta constants with
an analysis of the shape of CM period lattices
(Section~\ref{sec:boundonZ}).

The algorithm is basically the complex analytic method of
Spallek~\cite{spallek} and van Wamelen~\cite{vanwamelen},
and we show that it finishes in time
$\widetilde{O}(\Delta^{7/2})$, where $\Delta$ is the discriminant
of~$K$.
We give a complete running time analysis
of all parts of the algorithm,
and a proof of correctness including a rounding
error analysis.
We also provide various improvements
along the way.
\end{abstract}

\maketitle

\bibliographystyle{plain}

\section{Introduction}\label{sec:introduction}

The \emph{Hilbert class polynomial}
$H_K\in\ZZ[X]$ of an imaginary quadratic number field $K$
has as roots the $j$-invariants of complex
elliptic curves having complex multiplication (CM)
by the ring of integers of~$K$.
These roots generate the Hilbert class field of~$K$,
and Weber \cite{weber3} computed $H_K$
for many small~$K$.
The \emph{CM method} uses the reduction of $H_K$ modulo large primes to construct
elliptic curves over $\FF_p$ with a prescribed number of points,
for example for cryptography.
The bit size of $H_K$ grows like the discriminant
$\Delta$ of~$K$ (which is exponential
in the bit size of the input~$\Delta$)
and,
conditionally,
so does the running time of the algorithms
that compute it (\cite{enge,bbel}).

If we go from elliptic curves (genus~$1$) to genus-$2$ curves,
we get the \emph{Igusa class polynomials}
$H_{K,n}\in\QQ[X]$ ($n=1,2,3$) of a \emph{quartic CM-field}~$K$.
Their roots are the Igusa invariants of all complex genus-$2$ curves
having CM by the ring of integers of~$K$.
As in the case of genus~$1$, these roots generate class
fields and the reduction of Igusa class polynomials
modulo large primes~$p$ yields cryptographic
curves of genus~$2$.
Computing Igusa class polynomials is considerably more complicated
than computing Hilbert class polynomials.
Various algorithms
have been developed:
a complex analytic method by Spallek~\cite{spallek}
and van Wamelen~\cite{vanwamelen},
a $p$-adic method
%
\cite{ghkrw-2adic,ckl-3adic,carls-lubicz}
and a Chinese remainder method~\cite{eisentrager-lauter},
but no running time or precision bounds
were available.

This paper describes a complete and correct
algorithm that computes the Igusa class polynomials
$H_{K,n}\in\QQ[X]$ of quartic CM-fields~$K$,
i.e., fields that can be written
as
$K=\krootsmall$,
where 
$\Delta_0$ is a real quadratic fundamental discriminant and $a,b\in\ZZ_{>0}$ are
such that $-a+b\sqrt{\Delta_0}$
is totally negative.
Our algorithm is based on the complex analytic method.
The discriminant $\Delta$ of $K$ is of the form $\Delta=\Delta_1\Delta_0^2$
for a positive integer~$\Delta_1$.\label{not:delta1}
We may and will assume $0<a<\Delta$, as Lemma~\ref{denomintrinsic} below shows that
each quartic CM-field has
such a representation.
We disregard the degenerate
case of \emph{non-primitive} quartic CM-fields,
i.e.,
those that can be given with~$b=0$,
as abelian varieties with CM by non-primitive
quartic CM-fields
are isogenous to products of CM elliptic curves,
which can be obtained already using Hilbert class polynomials.
We prove the following unconditional running time bound for our algorithm.
We use $\smash{\widetilde{O}(g)}$ to mean
``at most $g$ times a polynomial in $\log g$''.
\begin{mainthm}\label{mainthmpage}
  Algorithm \ref{alg:mainalgorithm} computes $H_{K,n}$ ($n=1,2,3$)
  for any primitive quartic {CM-field}~$K$.
  It has a running time of
  $\widetilde{O}(\Delta_1^{7/2}\Delta_0^{11/2})$
  and the bit size of the output
  is $\widetilde{O}(\Delta_{1}^2\Delta_0^{3})$.
\end{mainthm}
We do not claim that the bound on our running time is optimal, but
an exponential running time is unavoidable, because the degree
of the Igusa class polynomials (as with Hilbert class polynomials)
is already bounded from below by a power of the discriminant.

\subsection*{Overview}

Section \ref{sec:igusa} provides a precise definition of the Igusa class polynomials
that we will work with, and mentions other definitions occurring
in the literature for which our main theorem is also valid.
In this section, we also propose the use of a new system
of absolute Igusa invariants,
which reduces the size
of the class polynomials.

Instead of enumerating curves,
it is easier to enumerate their Jacobians,
which are principally polarized abelian varieties (see Section~\ref{sec:avc}).
Van Wamelen \cite{vanwamelen} gave a method for enumerating
all isomorphism classes of principally polarized abelian varieties with CM
by a given CM-field. We give his results in
Section~\ref{sec:enumeratetriples} and improve them in
the sense that we list every polarized abelian variety only once.

Section \ref{sec:siegel} shows how principally polarized abelian varieties
give rise to points in the \emph{Siegel upper half space}~$\HH_2$.
These points are matrices known as
\emph{period matrices}.
Two period matrices correspond to isomorphic principally polarized
abelian varieties if and only if they are in the same orbit under the action of
the \emph{symplectic group}
$\mathrm{Sp}_{4}(\ZZ)$. In Section \ref{sec:up}, we
give a detailed analysis of a reduction algorithm
that replaces period matrices by $\mathrm{Sp}_4(\ZZ)$-equivalent
ones in a
\emph{fundamental domain}
$\calF_2\subset \HH_2$.

Absolute Igusa invariants can be computed from period matrices
by means of modular forms known as
\emph{(Riemann) theta constants}. Section~\ref{sec:theta}
introduces theta constants
and gives Igusa's formulas for expressing
Igusa invariants in terms of theta constants.
We propose the use of Igusa's formulas as a better way of
numerically computing Igusa invariants,
as they are much simpler than the formulas
in~\cite{hehcc18,spallek,weng}.
We then give bounds (based on the work of Dupont~\cite{dupont})
on the absolute values of theta constants and Igusa invariants
in terms of the entries of the reduced
period matrices.

In Section \ref{sec:boundonZ}, we 
give our upper bound on the entries of the reduced
period matrices.
This upper bound, together with 
the denominator bounds
of Goren and Lauter (Section~\ref{sec:denominators})
is the key to making a running time bound possible.

Section~\ref{sec:degree} bounds the degree of Igusa class polynomials
and Section~\ref{sec:invariants} motivates our choice of
invariants in Section~\ref{sec:igusa}.
Finally,
Section~\ref{sec:polyfromroots} explains how to reconstruct a rational polynomial from its complex
roots and Section \ref{sec:algorithm} puts everything together into
a single algorithm and a proof of the main theorem.

\begin{remark}
Our methods can also be applied to the case of elliptic curves,
though most steps are then overly complicated or unnecessary.
In fact, Theorem~\ref{thm:polynomialmultiplication} below, together
with the main results of Dupont~\cite{dupontarticle}, forms exactly
the missing rounding error analysis of Enge~\cite{enge}.
This shows that the main result of \cite{enge}, which bounds
the running time of computing Hilbert class polynomials, is valid
also without its heuristic assumption.
%
%
This is the first unconditional bound on the running time
of the computation of Hilbert class polynomials.
\end{remark}

\section{Igusa class polynomials}
\label{sec:igusa}

The \emph{Hilbert class polynomial}
of an imaginary quadratic number field $K$
is the polynomial of which the roots in $\CC$ are the \emph{$j$-invariants}
of the elliptic curves over $\CC$ with complex multiplication by
the ring of integers $\OK$ of~$K$.
For a genus-$2$ curve, one needs three \emph{absolute Igusa invariants}
$i_1$, $i_2$, $i_3$, instead of only one $j$-invariant, to fix its isomorphism class.

\subsection{Igusa invariants}

Let $k$ be a field of characteristic different from~$2$.
Any curve of genus~$2$ over~$k$,
i.e., a projective, smooth, geometrically irreducible, algebraic
curve over $k$ of which the genus is~$2$,
has an affine model of the form $y^2=f(x)$, where $f\in k[x]$ is a separable polynomial
of degree~$6$.
Let $\alpha_1,\ldots,\alpha_6$ be the six distinct roots of $f$ in $\overline{k}$,
and let $a_6$ be the leading coefficient.
For any permutation $\sigma\in S_6$, let $(ij)$ denote the
difference $(\alpha_{\sigma(i)}-\alpha_{\sigma(j)})$.
We can then define the
\emph{homogeneous Igusa-Clebsch invariants}
in compact notation that we explain below, as
\begin{align*}
  \homigu{2}{} &= a_6^2\sum_{15}(12)^2(34)^2(56)^2, &
  \homigu{4}{} &= a_6^4\sum_{10}(12)^2(23)^2(31)^2(45)^2(56)^2(64)^2,\\
  \homigu{10}{} &= a_6^{10}\prod_{i<j} (\alpha_{i}-\alpha_j)^2, &
  \homigu{6}{}
   &= a_6^6\sum_{60}(12)^2(23)^2(31)^2(45)^2(56)^2(64)^2(14)^2(25)^2(36)^2.
\end{align*}
The sum is taken over all distinct expressions (in the roots of~$f$)
that are obtained when $\sigma$ ranges over~$S_6$. The subscript indicates
the number of expressions encountered.
For example, for $\homigu{4}{}$
there are $10$ ways of partitioning the six roots of $f$
into two subsets of three, and each yields a summand
that is the product of two cubic discriminants.
For each of the $10$ ways of partitioning the six roots of $f$
into two subsets of three, there are $6$ ways of giving
a bijection between those two subsets, and each gives a
summand for~$\homigu{6}{}$.

The invariant $\homigu{10}{}$ is simply the discriminant of~$f$,
which is non-zero as $f$ is separable.
The invariants $I_2$, $I_4$, $I_6$, $I_{10}$ were introduced by Igusa~\cite{igusa}, who
denoted them by $A$, $B$, $C$, $D$ and based them on invariants of Clebsch.

By the symmetry in the definition, each of the homogeneous
invariants is actually a polynomial in the coefficients of~$f$,
hence an element of~$k$. Actually, we will use
another invariant, given by
$\homigu{6}{\prime} = \frac{1}{2}(I_2I_4-3I_6)$,
which is ``smaller'' than~$I_6$ in the sense that it is
a modular form of weight~6,  while $I_{6}$ is a quotient
of modular forms of weights 16 and 10 (see Section~\ref{sec:theta}).

We define the \emph{absolute Igusa invariants} by 
\[   i_1= \frac{\homigu{4}{}
                \homigu{6}{\prime}}{\homigu{10}{}},\quad
   i_2= \frac{\homigu{2}{}
                 \homigu{4}{2}}{\homigu{10}{}},\quad
   i_3= \frac{\homigu{4}{5}}{\homigu{10}{2}}.
\]
These are not equal to the absolute Igusa invariants in many other
articles, but see Section~\ref{ssec:invariants}.
The values of the absolute Igusa invariants of a curve $C$
depend only
on the $\overline{k}$-isomorphism
class of the curve~$C$. Conversely, for any triple $(i_1^0,i_2^0,i_3^0)$,
if $6$ and $i_3^0$ are non-zero in~$k$, then 
there exists a
curve $C$ of genus~$2$
(unique up to isomorphism)
over $\overline{k}$ with $i_n(C)=i_n^0$ ($n=1,2,3$),
and this curve can be constructed using an algorithm of Mestre~\cite{mestre}.
The case $i_3^0=0$
can be dealt with by using additional or modified absolute Igusa invariants
(e.g.~\cite{cardona-quer} and~\cite[III.5 equation (5.3)]{phdthesis}).
In case $6=0$, one needs to use other invariants of Igusa~\cite{igusa}.

\subsection{Igusa class polynomials}

\begin{definition}
Let $K$ be a primitive quartic CM-field.
   The \emph{Igusa class polynomials} of $K$ are the three polynomials
   $$H_{K,n}=\prod_{C} (X-i_n(C))\quad \in \QQ[X]\quad (n\in\{1,2,3\}),$$
   where the product ranges over the isomorphism classes of algebraic genus-$2$
   curves over $\CC$ of which the Jacobian has complex multiplication
   by~$\OK$.
\end{definition}
For the definitions of the Jacobian and complex
multiplication, see Section~\ref{sec:avc}.
We will see in Section~\ref{sec:enumeratetriples}
that the product in the definition
is indeed finite.
The polynomial is rational, because any conjugate of a CM curve has CM by the same ring.

\subsection{Alternative invariants}
\label{ssec:invariants}
In the literature, one finds various sets of absolute Igusa invariants.
For example, \cite{cardona-quer}, \cite{echidna}, \cite{igusa}, and \cite{spallek}
all make different choices.
The triple of invariants that seems most standard
in computations
is Spallek's 
$j_1=2^{-3}\homigu{2}{5}\homigu{10}{-1}$,
$j_2= 2\homigu{2}{3}\homigu{4}{}\homigu{10}{-1}$,
$j_3= 2^3 \homigu{2}{2}\homigu{6}{}\homigu{10}{-1}$
(occurring up to powers of~$2$ in 
\cite{hehcc18,goren-lauter,spallek,vanwamelen,weng,yang}).
We will show in Section~\ref{sec:invariants}
that our choice of absolute invariants yields smaller class polynomials,
both in experiments and in terms of the best proven upper bounds
for denominators and absolute values of coefficients.

If the base field $k$ has characteristic~$0$, then
Igusa's and Spallek's absolute invariants,
as well as most of the other invariants in the literature,
lie in the $\QQ$-algebra $A$
of homogeneous elements of degree $0$ of 
$\QQ[\homigu{2}{},\homigu{4}{},\homigu{6}{},\homigu{10}{-1}]$.
Our main theorem remains true if $(i_1,i_2,i_3)$
in the definition of the
Igusa class polynomials is replaced by any finite list of elements
of~$A$.

\subsection{Interpolation formulas}
\label{ssec:hhat}

If we take one root of each of the Igusa class polynomials,
then we get a triple of invariants and thus (if $2,3,i_3\not=0$)
an isomorphism class of curves of genus~$2$.
That way, the three Igusa class polynomials describe $d^3$
triples of invariants, where $d$ is the degree of the polynomials.
The $d$ triples corresponding to curves
with CM by $\OK$ are among them, but
the Igusa class polynomials give no means of telling which they are.

To solve this problem, (and thus greatly reduce the number of curves
to be checked during explicit CM constructions),
we use the following modified Lagrange
interpolation:
$$ \widehat{H}_{K,n}=\sum_{C}\left( i_n(C) \prod_{C'\not=C}
(X-i_1(C'))\right)\quad \in\QQ[X],\quad(n\in\{2,3\}).$$
If $H_{K,1}$ has no roots of multiplicity greater than~$1$, then
the triples of invariants corresponding to curves with CM by $\OK$
are exactly the triples $(i_1,i_2,i_3)$ such that
$$H_{K,1}(i_1)=0,\quad i_n =\frac{\widehat{H}_{K,n}(i_1)}{H_{K,1}'(i_1)}\quad (n\in\{2,3\}).$$
Our main theorem is also valid if we replace $H_{K,2}$ and $H_{K,3}$
by $\widehat{H}_{K,2}$ and~$\widehat{H}_{K,3}$.

This way of representing algebraic numbers (like our
$i_2$, $i_3$) in terms of others (our $i_1$)
appears in 
Hecke~\cite[Hilfssatz~a in \S36]{hecke},
and is sometimes called \emph{Hecke representation}
(e.g.~\cite{enge-morain}).
The idea to use this modified Lagrange interpolation
in the definition of Igusa class polynomials is due
to Gaudry, Houtmann, Kohel, Ritzenthaler,
and Weng~\cite{ghkrw-2adic}, who give
a heuristic argument that the height of
the polynomials $\smash{\widehat{H}_{K,n}}$ is smaller than
the height of the usual Lagrange
interpolation.

If ${H}_{K,1}$ has only double roots, then these interpolation formulas
are useless.
In practice, this never happens,
but the theoretical
possibility that it does happen
is handled in Section~III.5 of~\cite{phdthesis}.
There it is proven that our main result applies not just to
computing Igusa class polynomials, but also to computing
the CM-by-$K$ locus inside the coarse moduli space $\mathop{\mathrm{Spec}}(A)(\CC)$
of genus-2 curves over~$\CC$.

\section{Jacobians and complex multiplication}

\label{sec:avc}

\label{ssec:tori}

Instead of enumerating CM curves, we enumerate their 
\emph{Jacobians}.
We now quickly recall the definition from~\cite{birkenhake-lange}.
Given a smooth projective irreducible algebraic curve $C/\C$,
let $H^0(\omega_C)$ be the complex vector space of holomorphic
differential $1$-forms on $C$
and let $H^0(\omega_C)^\vee$ be its dual vector space.
Its dimension $g$ is the genus of~$C$
and our main result concerns the case $g=2$.
There is a canonical injection
of the homology group $H_1(C,\ZZ)$ 
into $H^0(\omega_C)^\vee$ (given by integration),
and the image is a lattice of rank $2g$.
The quotient complex torus $J(C)=H^0(\omega_C)^\vee/H_1(C,\ZZ)$
is the \emph{unpolarized Jacobian} of~$C$.

The \emph{endomorphism ring} $\mathrm{End}(\C^g/\Lambda)$ of a complex torus
$\C^g/\Lambda$ is the ring
of $\C$-linear endomorphisms of $\C^g$ that map $\Lambda$ to itself.
A \emph{CM-field} is a totally imaginary quadratic extension of a totally real number field.
We say that a complex torus $T$
of (complex) dimension $g$
has \emph{complex multiplication}
(or \emph{CM}) by an order $\mathcal{O}\subset K$
if $K$ has degree $2g$ and there exists an embedding
$\mathcal{O}\rightarrow \mathrm{End}(T)$.
We say that a curve $C$ has CM if $J(C)$ does.

It turns out that $J(C)$ is not just any complex torus, but that it comes
with a natural \emph{principal polarization}.
A polarization of a complex torus $\C^g/\Lambda$
is an alternating $\RR$-bilinear
form $E:\C^g\times \C^g\rightarrow \RR$ such that $E(\Lambda,\Lambda)\subset \Z$
holds
and
$(u,v)\mapsto E(iu,v)$ is symmetric and positive definite.
We call a polarization \emph{principal}
if its determinant with respect to
a $\Z$-basis of~$\Lambda$ is~$1$.
If we denote by $\cdot$ the anti-symmetric
intersection pairing on $H_1(C,\ZZ)$ extended
$\RR$-bilinearly to $H^0(\omega_C)^\vee$, then
$E:(u,v)\mapsto -u\cdot v$ defines a principal polarization on $J(C)$.
By the \emph{(polarized) Jacobian} of $C$, we mean
the torus together with this principal polarization.

A torus together with a (principal) polarization, such as the Jacobian of a curve,
is called a \emph{(principally) polarized abelian variety}.
An \emph{isomorphism} $f:(\C^g/\Lambda,E)\rightarrow (\C^g/\Lambda',E')$
of (principally)
polarized abelian varieties
is a $\C$-linear isomorphism $f:\C^g\rightarrow\C^g$ such that $f(\Lambda)=\Lambda'$
and $f^*E'=E$, where
$f^*E'(u,v)=E'(f(u),f(v))$ for all~$u,v\in\C^g$.

\begin{theorem}[{Torelli \cite[Thm.~11.1.7]{birkenhake-lange}}]
 Two algebraic curves over $\C$ are isomorphic
if and only if their Jacobians are isomorphic
(as polarized abelian varieties).\qed
\end{theorem}

The product of two polarized abelian varieties $(T_1,E_1)$ and $(T_2,E_2)$ 
has a natural polarization $(v,w)\mapsto E_1(v_1,w_1)+E_2(v_2,w_2)$ called the
\emph{product polarization}.

\begin{theorem}[Weil]\label{weilsthm}
Any principally polarized abelian surface over $\C$
is either a product of elliptic curves with the product polarization
or the Jacobian of a curve of genus~$2$.
\end{theorem}
\begin{proof}
This is \cite[Corollary 11.8.2]{birkenhake-lange}.
See also Remark~\ref{rem:constructiveweil} below.
\end{proof}

\section{Abelian varieties with CM}
\label{mainsec:classgroup}

\label{sec:enumeratetriples}

In this section, we give an algorithm
that computes a complete set of representatives of
the isomorphism classes of CM abelian varieties
for a given CM-field.

First, Section~\ref{ssec:quotientbyideal} shows how a CM abelian variety
is represented as a quotient of $\C^g$ by an ideal of~$K$.
Section~\ref{ssec:enumerategeneral} makes this into an
algorithm,
which works for CM-fields of arbitrary degree.
In Section~\ref{ssec:quarticcmfields}, we specialize to the case
of quartic CM-fields.
Finally, Section~\ref{sec:enumeratingdetails}
gives details needed for proving a bound on the running
time and the output size.

\subsection{CM abelian varieties as quotients by ideals}\label{ssec:quotientbyideal}

\label{cmintro:ssec:idealpol}
\label{cmintro:sec:triples}

Let $K$ be any CM-field of degree~$2g$.
A \emph{CM-type} of $K$ with values in $\C$ is a set
$\Phi = \{\phi_1,\ldots,\phi_g\}$ consisting
of one embedding $\phi_i:K\rightarrow \C$
for each complex conjugate pair of embeddings.
We identify $\Phi$ with the ring homomorphism
$K\rightarrow \C^g$ given by
$\Phi(\alpha) = (\phi_1(\alpha),\ldots,\phi_g(\alpha))$.
Let $\rho_{\Phi} : K \rightarrow \mathrm{End}_\C(\C^g): \alpha\mapsto \diag\Phi(\alpha)$.

We say that $\Phi$ is \emph{induced} from $K_1\subset K$
if $\{\phi_{|K_1} : \phi\in \Phi\}$ is a CM-type of $K_1$.
We say that $\Phi$ is \emph{primitive} if it is not
induced from a CM-subfield $K_1\not=K$.

Let $A=\C^g/\Lambda$ be an abelian variety with CM by an order
in a CM-field~$K$,
and let $\iota$ be an embedding $K\rightarrow \End(A)\otimes \Q$.
It is known (\cite[\S 5.2 in Chapter II]{shimura-taniyama})
that the composite map 
$\rho:K\rightarrow \mathrm{End}(A)\otimes \Q\rightarrow \mathrm{End}_\C(\C^g)$
equals $\rho_\Phi$ for some CM-type $\Phi$ and some choice of basis of~$\C^g$.
We say that $A$ is \emph{of type $\Phi$} with respect to~$\iota$.

Let $\mathcal{D}_{K/\Q}$ be the different of~$K$,
let $\mathfrak{a}$ be a  fractional
$\mathcal{O}_K$-ideal, and suppose
that there exists $\xi\in K$
such that $\xi\OK$
equals $(\mathfrak{a}\overline{\mathfrak{a}}\mathcal{D}_{K/\Q})^{-1}$
and
$\phi(\xi)$ lies on the positive imaginary axis for
every
$\phi\in \Phi$.
The map
$\Phi(K)\times\Phi(K)\rightarrow\Q$ given by
\begin{equation}\label{cmintro:eq:polofxi}
(\Phi(x),\Phi(y))\mapsto
\mathrm{Tr}_{K/\Q}(\xi\overline{x}y)\quad \mbox{for $x,y\in K$}
\end{equation}
can be extended uniquely
to an
$\RR$-bilinear form
$E=E_{\Phi,\xi}:\C^g\times \C^g\rightarrow \RR$.

\begin{theorem}\label{cmintro:firstcharacterization}\label{cmintro:thm:triplessimple}
Suppose $\Phi$ is a CM-type of a 
CM-field $K$ of degree $2g$.
Then the following holds.
\begin{enumerate}
\item \label{cmintro:itm:firstchar1}
For any triple $(\Phi,\A,\xi)$ as above,
the pair $(\C^g/\Phi(\A),E)$ defines a 
principally polarized abelian variety $A({\Phi,\A,\xi})$
with CM by~$\OK$
of type~$\Phi$.
\item \label{cmintro:itm:firstchar2} Every principally polarized abelian variety over $\C$
with CM by $\OK$ of type $\Phi$
is isomorphic to $A({\Phi,\A,\xi})$ for some triple $(\Phi,\A,\xi)$ as
in part~\ref{cmintro:itm:firstchar1}.
\item \label{cmintro:itm:primitive} The abelian variety $A(\Phi,\A,\xi)$ is
simple if and only if $\Phi$ is primitive. If this is the case,
then the embedding $\iota:K\rightarrow \mathrm{End}(A)\otimes \Q$
is an isomorphism.
\item \label{cmintro:itm:whenisomorphicwithoutphi}
Let $(\Phi,\A,\xi)$ and
$(\Phi,\A',\xi')$ be triples as above with the same CM-type~$\Phi$.
If
there exists
$\gamma\in K^*$ such that
\begin{enumerate}
\item \label{cmintro:itm:whenisom1}$\A'=\gamma\A$ and
\item \label{cmintro:itm:whenisom2} $\xi'=(\gamma\overline{\gamma})^{-1}\xi$,
\end{enumerate}
then the principally polarized
abelian varieties $A({\Phi,\A,\xi})$ and $A({\Phi,\A',\xi'})$
are isomorphic.
If $\Phi$ is primitive, then the converse holds.
\end{enumerate}
\end{theorem}
\begin{proof}
This result can be derived from
Shimura-Taniyama~\cite{shimura-taniyama},
and first appeared in a form similar to the above in
Spallek~\cite[S\"atze 3.13, 3.14, 3.19]{spallek}.
See
van Wamelen~\cite[Thms.~1, 3,~5]{vanwamelen} for a
detailed published proof.
\end{proof}
\begin{definition}
We call two triples $(\Phi,\A,\xi)$ and $(\Phi,\A',\xi')$
with the same
type \emph{equivalent} if there exists $\gamma\in K^*$
as in (\ref{cmintro:itm:whenisomorphicwithoutphi}) of
Theorem~\ref{cmintro:firstcharacterization}.
\end{definition}
Let $K$ be any {CM-field} with maximal totally real subfield~$K_0$.
Let $h$ (resp.\ $h_0$) be the class number of $K$ (resp.\ $K_0$)
and let $h_1=h/h_0$.
\begin{proposition}\label{cmintro:prop:countav}
  The number of pairs $(\Phi,A)$, where $\Phi$ is a
  CM-type and $A$ is an isomorphism
  class of abelian varieties over $\C$ with CM by $\OK$ of type $\Phi$, is
  $$h_1\cdot \# \OKO^*/N_{K/K_0}(\OK^*).$$
\end{proposition}
\begin{proof}
  Let $I$ be the group of invertible $\OK$-ideals and $S$ the set of pairs
  $(\A,\xi)$ with $\A \in I$ and $\xi\in K^*$ such that
  $\xi^2$ is totally negative and
  $\xi\OK=(\A\overline{\A}\mathcal{D}_{K/\Q})^{-1}$.
  The group $K^*$ acts on $S$ via $x(\A,\xi)=(x\A,x^{-1}\overline{x}^{-1}\xi)$
  for $x\in K^*$.
  By Theorem~\ref{cmintro:firstcharacterization},
  the set that we need to count is in bijection with
  the set $K^*\backslash S$ of orbits.

  The fact that $S$ is non-empty is \cite[Thm.~4]{vanwamelen}.
  We give a shorter proof here.
  Let $z\in K^*$ be such that $z^2$ is a totally negative element of $K_0$.
  Note that $z\mathcal{D}_{K/K_0}=
  (z(\alpha-\overline{\alpha}) : \alpha\in\mathcal{O}_K)$
  is generated by elements of $K_0$, hence is
  of the form $\mathfrak{h}\mathcal{O}_K$ for some fractional
  $\OKO$-ideal $\mathfrak{h}$.
  The
  norm map $N_{K/K_0}:\mathrm{Cl}(K)\rightarrow\mathrm{Cl}(K_0)$
  is surjective because infinite primes ramify in~$K/K_0$
  (see \cite[Thm.~10.1]{washington-cycl}).
  In particular, there exist an element $y\in K_0^*$
  and a fractional $\OK$-ideal $\A_0$ such that
  $y\A_0\overline{\A_0}=\mathfrak{h}^{-1}\mathcal{D}_{K_0/\Q}^{-1}
   = z^{-1}\mathcal{D}_{K/\Q}^{-1}$ holds,
  so $(\A_0,yz)$ is an element of~$S$.

  Let $S'$ be the group of pairs $(\B,u)$, consisting of a fractional
  $\OK$-ideal $\B$ and a generator $u\in K_0^*$ of
  $\B\overline{\B}$.
  The group $K^*$ acts on $S'$ via $x(\B,u)=(x\B,x\overline{x}u)$ for $x\in K^*$,
  and we denote the group of orbits by $C=K^*\backslash S'$.
  The map $C\rightarrow K^*\backslash S:(\B,u)\mapsto (\B\A_0,u^{-1}yz)$ is a bijection
  and the sequence
\[  0 \longrightarrow
  \OKO^*/N_{K/K_0}(\OK^*) \displaystyle{\mathop{\longrightarrow}_{{u\mapsto(\OK,u)}}}
  C \displaystyle{\mathop{\longrightarrow}_{{(\B,u)\mapsto\B}}}
  \mathrm{Cl}(K) \displaystyle{\mathop{\longrightarrow}_{N}}
  \mathrm{Cl}(K_0)\longrightarrow
  0
\]
is exact, so $K^*\backslash S$ has the correct order.
\end{proof}
\begin{remark}
  The existence statement of Proposition~\ref{cmintro:prop:countav}
  contradicts the first remark below Proposition~1 of~\cite{deshalit-goren}.
  It turns out that that remark is false, and it follows
  that the supporting ``example''
  in~\cite{deshalit-goren}
  does not exist.
  That is,
  if $F$ is real quadratic with class number~$1$
  and a fundamental unit of norm~$1$,
  then there is no cyclic quartic CM-field containing~$F$.
\end{remark}

Theorem~\ref{cmintro:firstcharacterization} tells us exactly
when two CM varieties with the same CM-type are isomorphic.
The following two lemmas show what to do
when the CM-types are distinct, thus answering
a question of van Wamelen~\cite{vanwamelen}.
\begin{lemma}\label{cmintro:lem:differenttypes1}
  For any triple $(\Phi,\A,\xi)$ as above
  and $\sigma\in\mathrm{Aut}(K)$,
  we have $$A(\Phi,\A,\xi)\cong A(\Phi\circ\sigma,\sigma^{-1}(\A),\sigma^{-1}(\xi)).$$
\end{lemma}
\begin{proof}
  We find twice the same complex torus
  $\C^g/\Phi(\A)$.
  The first has polarization
  \begin{equation}\label{cmintro:eq:againe}
  E:(\Phi(\alpha),\Phi(\beta))\mapsto \Tr_{K/\Q}(\xi\overline{\alpha}\beta)
  \end{equation}
  for $\alpha,\beta\in \A$
  while the polarization of the second sends
  $(\Phi(\alpha),\Phi(\beta))$
  to $\Tr_{K/\Q}(\sigma^{-1}(\xi\overline{\alpha}\beta))$,
  which equals the right hand side of~\eqref{cmintro:eq:againe}.
\end{proof}
\begin{definition}
 We call two CM-types $\Phi$ and $\Phi'$ of~$K$ \emph{equivalent}
 if there exists $\sigma\in\Aut(K)$ with $\Phi'=\Phi\circ\sigma$.
\end{definition}
\begin{lemma}\label{cmintro:lem:differenttypes2}
Suppose $A$ and $B$ are abelian varieties
over $\C$ with CM by $K$ of 
types $\Phi$ and~$\Phi'$.
If $\Phi$ is primitive and not equivalent to $\Phi'$,
then $A$ and $B$ are not isogenous.
In particular, they are not isomorphic.
\end{lemma}
\begin{proof}
Suppose $f:A\rightarrow B$ is an isogeny.
It induces an isomorphism
$\varphi:\mathrm{End}(B)\otimes \Q\rightarrow \mathrm{End}(A)\otimes \Q$
given by $\varphi(g) = f^{-1}gf$.
Let $\iota_A:K\rightarrow \mathrm{End}(A)\otimes \Q$
and $\iota_B:K\rightarrow \mathrm{End}(B)\otimes \Q$
be embeddings of types $\Phi$ and~$\Phi'$.
Let $\sigma = \iota_A^{-1} \circ \varphi\circ \iota_B^{\vphantom{-1}}$
(where $\iota_A$ is an isomorphism by
Theorem~\ref{cmintro:firstcharacterization}.\ref{cmintro:itm:primitive}
because $\Phi$ is primitive).
Then $(A,\iota_A\circ \sigma)$ and $(B,\iota_B)$
have types $\Phi\circ \sigma$ and $\Phi'$.
As $f$ induces an isomorphism of the vector spaces of which $A$ and $B$ are quotients,
these CM-types are equal,
so $\Phi$ and $\Phi'$ are equivalent.
\end{proof}
\begin{definition}
We call two triples $(\Phi,\A,\xi)$ and $(\Phi',\A',\xi')$
\emph{equivalent} if there is an automorphism
$\sigma\in\mathrm{Aut}(K)$ such that 
$\Phi\circ\sigma=\Phi'$ holds and
$(\Phi,\sigma(\A'),\sigma(\xi'))$ is equivalent
to $(\Phi,\A,\xi)$
as in our definition below Theorem~\ref{cmintro:firstcharacterization}.
\end{definition}
\begin{proposition}
Given $(\Phi,\A,\xi)$ and $(\Phi',\A',\xi')$,
assume that $\Phi$ primitive.
Then we have $A(\Phi,\A,\xi)\cong A(\Phi',\A',\xi')$
if and only if $(\Phi,\A,\xi)$ and $(\Phi',\A',\xi')$
are equivalent.
\end{proposition}
\begin{proof}
This follows
from Theorem~\ref{cmintro:firstcharacterization}.\ref{cmintro:itm:whenisomorphicwithoutphi}
and Lemmas \ref{cmintro:lem:differenttypes1} and~\ref{cmintro:lem:differenttypes2}.
\end{proof}

\subsection{The algorithm}
\label{ssec:enumerategeneral}
\begin{algorithm}\label{alg:vwamelenalgorithm}\hfil\break
\textbf{Input:} A CM-field $K$ with maximal
totally real subfield $K_0$ such that $K$ does not
contain a strict CM-subfield.\\
\textbf{Output:} A complete set of representatives for
  the equivalence classes of principally polarized abelian varieties over~$\C$
  with CM by~$\OK$, each given by a triple $(\Phi,\A,\xi)$
  as in Theorem~\ref{cmintro:firstcharacterization}.
\begin{enumerate}
 \item Determine a complete set of representatives $T$
  of the set of equivalence classes of CM-types of $K$ with values in~$\C$.
\item \label{takeu} Determine a complete set of
  representatives $W$ of the quotient $$\OK^*/N_{K/K_0}(\OK^*).$$
\item Determine a complete set of representatives $I$ of the ideal class group of~$K$.
\item \label{takexizero} Take those $\A$ in $I$ such that $(\A\overline{\A}\mathcal{D}_{K/\Q})^{-1}$ is principal.
      For each, take a generator~$\xi_0$.
\item \label{takeaxi} For every pair $(\A,\xi_0)$ and every $w\in W$
      such that $\xi=w\xi_0$ is totally imaginary, take
     the CM-type $\Phi$ consisting of those embeddings of $K$
     into $\C$ that map $\xi$ to the positive
     imaginary axis. Output the triple $(\Phi,\mathfrak{a},\xi)$
     if $\Phi$ is in~$T$.
\end{enumerate}
\end{algorithm}
\begin{proof}
  By Theorem~\ref{cmintro:firstcharacterization}.\ref{cmintro:itm:firstchar1},
  the output consists only of principally polarized abelian varieties
  with CM by~$\OK$.
  Conversely, by Theorem~\ref{cmintro:firstcharacterization}.\ref{cmintro:itm:firstchar2}, 
  every principally polarized abelian variety $A$
  with CM by $\OK$ is isomorphic to $A({\Phi,\A,\xi})$
  for some triple $(\Phi,\A,\xi)$, and we will show now that
  such a triple is found exactly once by the algorithm.
  
  By Lemmas \ref{cmintro:lem:differenttypes1}
  and~\ref{cmintro:lem:differenttypes2},
  the CM-type $\Phi$ is unique exactly up to equivalence of CM-types.
  This uniquely determines $\Phi$ in~$T$.
  
  By Theorem~\ref{cmintro:firstcharacterization}.\ref{cmintro:itm:whenisomorphicwithoutphi},
  the ideal class of $\A$ is then uniquely determined,
  hence we find a unique $\A\in I$.
  The class of $\xi$ modulo $N_{K/K_0}(\OK^*)$ is uniquely determined
  by Theorem~\ref{cmintro:firstcharacterization}.\ref{cmintro:itm:whenisomorphicwithoutphi},
  hence so is $\xi$ as found in the algorithm.
\end{proof}
\begin{remark}
 Algorithm \ref{alg:vwamelenalgorithm} is basically
 Algorithm~1 of van Wamelen
 \cite{vanwamelen}
with the difference that we do not have
any duplicate abelian varieties.
\end{remark}

\subsection{Quartic CM-fields}
\label{ssec:quarticcmfields}

We now describe, in the quartic case,
the sets $T$ and $W$ of Algorithm~\ref{alg:vwamelenalgorithm},
and the
number of isomorphism classes of 
principally polarized CM abelian surfaces.

\begin{lemma}[Example 8.4(2) of \cite{shimura-taniyama}]\label{cmintro:lem:geq2cmtypes}
  Let $K$ be a quartic CM-field with a CM-type $\Phi=\{\phi_1,\phi_2\}$.
  Exactly one of the following holds.
\begin{enumerate}
 \item $K/\Q$ is Galois with Galois group $C_2\times C_2$
      and each CM-type of $K$ is non-primitive
      and induced from an imaginary quadratic subfield of $K$,
\item $K/\Q$ is cyclic Galois, and
      all four CM-types are equivalent and primitive,
\item $K/\Q$ is non-Galois, its
      normal closure has Galois group~$D_4$,
      each CM-type is primitive, and the equivalence
      classes of CM-types are $\{\Phi,\overline{\Phi}\}$
      and $\{\Phi',\overline{\Phi'}\}$ with
      $\Phi' = \{\phi_1,\overline{\phi_2}\}$.\qed
\end{enumerate}
\end{lemma}
Note that in particular, either all CM-types are primitive or none of them are.
This is why we use the word \emph{(non-)primitive}
also for the quartic CM-fields themselves.

Lemma \ref{cmintro:lem:geq2cmtypes} shows that we 
can take the set
$T$ to consist of a single CM-type
if $K$ is cyclic and we can take
$T=\{\Phi,\Phi'\}$ if $K$ is non-Galois.
\begin{lemma}\label{lem:generatorsofuplusses1}
  If $K$ is a primitive quartic CM-field, then
  $$\OK^*=\mu_K\OKO^*\quad\mbox{and}\quad N_{K/K_0}(\OK^*)=(\OKO^*)^2,$$
  where $\mu_K\subset\OK^*$ is the group of roots of unity, which has order $2$ or~$10$.
\end{lemma}
\begin{proof}
  As $K$ has degree $4$ and does not contain a primitive third or fourth root of unity,
  it is either $\Q(\zeta_5)$ or does not contain a root of unity different from~$\pm 1$.
  This proves that $\mu_K$ has order $2$ or~$10$. A direct computation shows that the lemma
  is true for $K=\Q(\zeta_5)$, so we assume~$\mu_K=\{\pm 1\}$.

  Note that the second identity follows from the first, so we only need to prove the first.
  Let~$\epsilon$ (resp.~$\epsilon_0$) be a generator of
  $\mathcal{O}_{K}^*$ (resp.\ $\mathcal{O}_{K_0}^*$) modulo $\langle -1\rangle$.
  Then without loss of generality, we have $\epsilon_0=\epsilon^k$ for some 
  positive integer~$k$. By taking norms $N_{K/K_0}$ on both sides,
  we find $\epsilon_0^2 = (\pm \epsilon_0^{l})^k$ for some integer~$l$,
  so $k\in\{1,2\}$.

  Suppose $k=2$.
  As $K=K_0(\sqrt{\epsilon_0})$ is a CM-field,
  we find that $\epsilon_0$ is totally negative,
  and hence $\epsilon_0^{-1}$ is the quadratic
  conjugate of~$\epsilon_0$ over~$\QQ$.
  Let
  $x=\epsilon-\epsilon^{-1}\in K$.
  Then $x^2=-2+\epsilon_0+\epsilon_0^{-1}=-2+\mathrm{Tr}(\epsilon_0)\in\Z$
  is negative, so $\Q(x)\subset K$ is
  imaginary quadratic, contradicting primitivity of~$K$.
  We conclude $k=1$, so $\mathcal{O}_K^* = \mathcal{O}_{K_0}^*$.
\end{proof}
In particular, we can take $W=\mu_K\cup \epsilon\mu_K$ for a fundamental
unit $\epsilon$ of $\OKO^*$.

\begin{lemma}\label{lem:countas}
  Let $K$ be a quartic CM-field.
  If $K$ is cyclic, then there are $h_1$ isomorphism classes of
  principally polarized abelian surfaces
  with CM by~$\OK$.
  If $K$ is non-Galois, then there are $2h_1$ such isomorphism classes.
\end{lemma}
\begin{proof}
  Proposition \ref{cmintro:prop:countav}
  gives $h_1\cdot \# \OKO^*/N_{K/K_0}(\OK^*)$ classes,
  but counts every abelian variety twice if $K$ is non-Galois
  and four times if $K$ is cyclic Galois
  (see Lemma~\ref{cmintro:lem:geq2cmtypes}).
  Next, 
  Lemma~\ref{lem:generatorsofuplusses1} shows
  $\# \OKO^*/N_{K/K_0}(\OK^*)=4$.
\end{proof}

\subsection{Implementation details}
\label{sec:classgroups}\label{sec:enumeratingdetails}

In practice, Algorithm \ref{alg:vwamelenalgorithm} takes
up only a very small portion of the time needed for Igusa class polynomial
computation.
The purpose of this section is to show that, for primitive quartic CM-fields, indeed
Algorithm \ref{alg:vwamelenalgorithm} can be run in time
$\widetilde{O}(\Delta)$
and to show that the size of the output for each isomorphism class is small:
only polynomial in $\log \Delta$.

It is well known that lists of representatives
for the class groups of number fields $K$ of fixed
degree can be computed in time $\widetilde{O}(\abs{\Delta}^{\frac{1}{2}})$,
where $\Delta$ is the discriminant of~$K$.
For details, see~\cite{schoof}.
The representatives of the ideal classes that are given in the output
are integral ideals of norm below the Minkowski bound,
which is $3/(2\pi^2)\abs{\Delta}^{1/2}$ for a quartic CM-field.

The algorithms in~\cite{schoof} show that
for each $\A$, we can check in time $\widetilde{O}(\abs{\Delta}^{\frac{1}{2}})$
whether $\A\overline{\A}\mathcal{D}_{K/\Q}$ is principal and, if so, write
down a generator~$\xi$.
The sets~$T$ and~$W$ are given
in Section~\ref{ssec:quarticcmfields},
where the fundamental unit~$\epsilon$
is a by-product of the class group computations.
In particular, it takes time at most $\widetilde{O}(\abs{\Delta})$ to perform all the steps
of Algorithm~\ref{alg:vwamelenalgorithm}.

A priori, the bit size of $\xi$ can be as large as the
regulator of~$K$, but we can easily make it much smaller as follows.
Identify
$K\otimes \RR$ with
$\C^2$
via the embeddings $\phi_1$, $\phi_2$ in the CM-type~$\Phi$,
and consider the standard Euclidean norm on $\C^2$.
Then find a short vector
$$b\abs{\xi}^{-1/2}=\left(\phi_1(b)\abs{\phi_1(\xi)}^{-1/2}, \phi_2(b)\abs{\phi_2(\xi)}^{-1/2}\right)$$
in the lattice $\OK\abs{\xi}^{-1/2}\subset \C^2$
and replace $\A$ with $b\A$ and $\xi$ with $(b\overline{b})^{-1}\xi$.
To find this short vector, we use
a version of the LLL-algorithm
that is quasi-linear in the bit size of the input
for fixed dimension, as in~\cite{eisenbrand-rote}.

By part \ref{cmintro:itm:whenisomorphicwithoutphi} of
Theorem~\ref{cmintro:firstcharacterization},
the change of $(\A,\xi)$ to $(b\A, (b\overline{b})^{-1}\xi)$
does not change
the corresponding isomorphism class of principally polarized abelian varieties.
This also does not change the fact that~$\xi^{-1}$ is in~$\OK$ and
that~$\A$ is an integral ideal.
Finally, we compute an LLL-reduced basis of $\A\subset \OK\otimes \RR=\C^2$.
We get the following result.
\begin{lemma}\label{lem:62}
  If we run Algorithm~\ref{alg:vwamelenalgorithm}
  in the way we have just described,
  then 
  on input of a primitive quartic {CM-field}~$K$, given
  as $K=\smash{\krootbig}$ for integers $a$, $b$, $\Delta_0$
  with $0<a<\Delta$, it takes
  time $\widetilde{O}(\Delta)$.
  For each triple $(\Phi,\A,\xi)$ in the output,
  the ideal~$\A$ is given by an LLL-reduced basis, and both $\xi\in K$
  and the basis of~$\A$ have bit size $O(\log\Delta)$.
\end{lemma}
\begin{proof}
First, compute the ring of integers $\mathcal{O}_K$ of $K$
using the algorithm of Buchmann and Lenstra~\cite{buchmann-lenstra}.
This takes polynomial time plus the time needed to factor the
discriminant of the defining polynomial of~$K$,
which is small enough because of the assumption $0<a<\Delta$.
Then do the class group computations as explained above.

For each triple $(\Phi,\A,\xi)$, before we apply the LLL-reduction,
we can assume that $\A$ is an integral ideal of 
norm below the Minkowski bound,
hence we have $$N_{K/\Q}(\xi^{-1})=N_{K/\Q}(\A)^2N_{K/\Q}
(\mathcal{D}_{K/\Q})\leq C\Delta^3$$ for some constant~$C$.

The covolume of the lattice
$\abs{\xi}^{-1/2}\OK\subset \OK\otimes \RR=\C^2$
is
$N_{K/\Q}(\xi^{-1})\Delta^{1/2}$,
so we find a vector
$b\abs{\xi}^{-1/2}\in\abs{\xi}^{-1/2}\OK$
of length $\leq C'(N_{K/\Q}(\xi^{-1})\Delta)^{1/8}$
for some constant $C'$.
In particular, $b\overline{b}\xi^{-1}$ has all
absolute values below ${C'}^2N_{K/\Q}(\xi^{-1})^{1/4}\Delta^{1/4}$.
Therefore, $b\overline{b}\xi^{-1}$ has bit size $O(\log \Delta)$ and norm at most
${C'}^{8}N_{K/\Q}(\xi^{-1})\Delta$, so $b$ has norm at most ${C'}^{4}\Delta^{1/2}$.

This implies that $b\A$ has norm at most $C''\Delta$,
so an LLL-reduced basis
has a bit size that is $O(\log({\mathrm{covol}}(b\A)))=O(\log \Delta)$.

All elements $x\in K$ that we encounter
can be given (up to multiplication by units in $\OKO^*$)
with all absolute values below $\sqrt{N_{K/\Q}(a)}\abs{\epsilon}$.
Therefore,
the bit size of the numbers that are
input to the LLL-algorithm is
$\widetilde{O}(\mathrm{Reg}_K)=\widetilde{O}(\Delta^{1/2})$,
hence every execution
of the LLL algorithm takes time only
$\widetilde{O}(\Delta^{1/2})$ for each ideal class.
\end{proof}

\section{Symplectic bases}
\label{sec:siegel}

\subsection{Symplectic bases, period matrices, and the action of the
symplectic group}\label{ssec:periodmatrices}

  Let $(\C^g/\Lambda,E)$ be a principally polarized abelian variety.
  For any basis $b_1,\ldots,b_{2g}$ of $\Lambda$,
  we associate to the form $E$ the matrix
  $N=(n_{ij})_{ij}\in\mathrm{Mat}_{2g}(\ZZ)$
  given by $n_{ij}=E(b_i,b_j)$.
  We say that $E$ is given with respect to the basis $b_1,\ldots,b_{2g}$
  by the matrix~$N$.
  
  The lattice $\Lambda$ has a basis that is \emph{symplectic}
  with respect to~$E$, i.e., 
  a $\Z$-basis $e_1,\ldots, e_g,v_1,\ldots,v_g$ with respect to which $E$ is given
  by the matrix $\Omega$, given in terms of $(g\times g)$-blocks as
  \begin{equation}
  \label{eq:omega}
    \Omega=\tbt{ 0 & 1_g \\ -1_g & 0}.
    \end{equation}
  The vectors $v_i$ form a $\C$-basis of $\C^g$ and if we rewrite $\C^g$
  and $\Lambda$ in terms
  of this basis, 
  then $\Lambda$ becomes $Z\Z^g+\ZZ^g$, where $Z$ is a
  \emph{period matrix},
  i.e., a symmetric matrix over $\C$
  with positive definite imaginary part.
  The set of all $g\times g$ period matrices is called
  the \emph{Siegel upper half space}
  and denoted by~$\HH_g$.
  It is a subset of the Euclidean
  $2g^2$-dimensional real vector
  space $\mathrm{Mat}_{g}(\C)$.

There is an action on this space by the
\emph{symplectic group}
$$\mathrm{Sp}_{2g}(\ZZ)=\{M\in\mathrm{GL}_{2g}(\ZZ): M\transpose  \Omega M=\Omega\}\subset\mathrm{GL}_{2g}(\ZZ),$$
given in terms of $(g\times g)$-blocks by
$$\tbt{A & B \\ C & D}(Z)=(AZ+B)(CZ+D)^{-1}.$$
The association of $Z$ to $(\C^g/Z\Z^g+\ZZ^g,E)$ gives a bijection
between the set $\mathrm{Sp}_{2g}(\ZZ)\backslash\HH_g$
of orbits  and the set of principally
polarized abelian varieties over $\C$ up to isomorphism.

\subsection{Finding a symplectic basis for \texorpdfstring{{$\Phi(\A)$}}{Phi(a)}}
\label{ssec:computesympbasis}

Now it is time to compute symplectic bases.
In Algorithm \ref{alg:vwamelenalgorithm}, we computed
a set of abelian varieties over $\C$, each given
by a triple $(\Phi,\A,\xi)$, where $\A$ is an ideal in $\OK$, given by
a basis,
$\xi$ is in $K^*$ and $\Phi$ is a CM-type of~$K$.
We identify $\A$ with the lattice $\Lambda=\Phi(\A)\subset\C^g$
and recall that the bilinear form $E:\A\times\A\rightarrow\Z$ 
is given by $E:(x,y)\mapsto\mathrm{Tr}_{K/\Q}(\xi \overline{x}y)$.
We can write down the matrix $N\in\mathrm{Mat}_{2g}(\ZZ)$ of $E$
with respect
to the basis of~$\A$.
Computing a symplectic basis of $\A$ then comes down to computing
a change of basis $M\in\mathrm{GL}_{2g}(\ZZ)$
of $\A$ such that $M\transpose NM =\Omega$, with $\Omega$ as in~\eqref{eq:omega}.
This is done by the following
algorithm.

\begin{algorithm}\label{alg:gramschmidt}\hfil\break
\textbf{Input:} A matrix $N\in\mathrm{Mat}_{2g}(\ZZ)$ such that $N\transpose =-N$ and $\det N=1$.\\
\textbf{Output:} $M\in\mathrm{GL}_{2g}(\ZZ)$ satisfying $M\transpose NM =\Omega$.\\
For $i=1,\ldots,g$, do the following.
\begin{enumerate}
\item Let $e_i'\in\Z^{2g}$ be a vector
      linearly independent of $$\{e_1,\ldots,e_{i-1},v_1,\ldots,v_{i-1}\}.$$
\item From $e_i'$, compute the following vector~$e_i$,
      which is orthogonal to $e_1,\ldots,e_{i-1}$, $v_1,\ldots,v_{i-1}$
      with respect to~$N$:
$$\textstyle{e_i=\frac{1}{k}\left(e_i'-\sum_{j=1}^{i-1} (e_j\transpose  N e_i')v_j+\sum_{j=1}^{i-1}(v_j\transpose  N e_i')e_j\right),}$$
where $k$ is the largest positive integer such that the resulting vector $e_i$ is in $\Z^{2g}$.
\item \label{itm:nonobviousstep3} Let $v_i'$ be such that $e_i\transpose  N v_i'=1$.
We will explain this step below.
\item From $v_i'$, compute the following vector~$v_i$,
      which is orthogonal to $e_1,\ldots,e_{i-1}$, $v_1,\ldots,v_{i-1}$
      with respect to~$N$
      and satisfies $e_i\transpose Nv_i=1$:
      $$\textstyle{v_i=v_i'-\sum_{j=1}^{i-1} (e_j\transpose  N v_i')v_j+\sum_{j=1}^{i-1}(v_j\transpose  N v_i')e_j.}$$
\end{enumerate}
Output the matrix $M$ with columns $e_1,\ldots,e_g,v_1,\ldots,v_g$.
\end{algorithm}

Existence of $v_i'$ as in
step~\ref{itm:nonobviousstep3}
follows from the facts that $N$ is invertible and that
$e_i\in\Z^{2g}$ is not divisible by integers greater than~$1$.
Actually finding $v_i'$ means finding
a solution of a linear equation over $\Z$,
which can be done using the LLL-algorithm as
in~\cite[Section 14]{lattices}.

If we apply the Algorithm~\ref{alg:gramschmidt} to the matrix
$N$ mentioned above it, then the output matrix $M$
is a basis transformation that yields a symplectic
basis of $\Lambda$ with respect to~$E$.
For fixed~$g$, Algorithm~\ref{alg:gramschmidt} takes time polynomial
in the size of the input, hence polynomial time in the bit
sizes of $\xi\in K$ and the basis of~$\A$.
Lemma \ref{lem:62} tells
us that for $g=2$, we can make sure that
both $\xi\in K$ and the basis of $\A$ have a bit
size that is polynomial in $\log\Delta$, so obtaining a
period matrix $Z$ from a triple $(\Phi,\A,\xi)$
takes time only polynomial in $\log\Delta$.
This implies also that the bit size of $Z$ (as a matrix
with entries in a number field) is polynomial in $\log\Delta$.

\section{The fundamental domain of the Siegel upper half space}
\label{sec:up}

In the genus-$1$ case, to compute the $j$-invariant of a point $z\in \HH=\HH_1$,
one should first move $z$ to the \emph{fundamental domain}
for $\mathrm{SL}_2(\ZZ)$,
or at least away from $\Im z=0$, to get good convergence.
We use the term \emph{fundamental domain}
loosely,     
meaning a connected subset $\mathcal{F}$ of $\HH_g$ such
that every $\mathrm{Sp}_{2g}(\ZZ)$-orbit
has a representative in $\mathcal{F}$,
and that this representative is unique, except
possibly if it is on the boundary of~$\mathcal{F}$.

In genus~$2$, when computing $\theta$-values at a point $Z\in \HH_2$,
as we will do in Section \ref{sec:theta}, we move the point
to the fundamental domain for~$\mathrm{Sp}_4(\ZZ)$.

We will treat the genus-1 case first, not only because
of the analogy, but also because the reduction
algorithm for the genus-1 case is part of the reduction
algorithm for the genus-2 case.

\subsection{The genus-1 case}

For $g=1$, the fundamental domain $\calF\subset \HH$ is the set of
 $z=x+iy\in\HH$ that satisfy
\begin{enumerate}
\item[(F1)]
$-\frac{1}{2}< x \leq \frac{1}{2}$ and
\item[(F2)] $\abs{z}\geq 1$.
\end{enumerate}
One usually adds a third condition $x\geq 0$ if $\abs{z}=1$
in order to make the orbit representatives unique,
but we will omit that condition
as we allow boundary points of $\mathcal{F}$ to be non-unique
in their orbit.
To move $z$ into this fundamental domain, we simply iterate the following
until $z=x+iy$ is in $\calF$:
\begin{equation}
 \label{eq:tofunddom1}
\begin{array}{rl}
 1. & \quad z\leftarrow z+\lfloor -x+\frac{1}{2}\rfloor,\\
 2. & \quad z\leftarrow -z^{-1}$ if $\abs{z}<1.
\end{array}
\end{equation}
We will see in Lemma~\ref{lem:gaussreduction}
that this procedure termintes.
We first phrase it in terms of positive definite
$(2\times 2)$-matrices $Y\in\mathrm{Mat}_2(\RR)$,
which will come in handy in the genus-$2$ case.
We identify such a matrix 
$$Y=\tbt{y_1 & y_3 \\ y_3 & y_2}$$
with the positive definite binary quadratic form
$f=y_1 X^2+2y_3 XY+y_2 Y^2\in\RR[X,Y]$.
Let $\phi$ be the map that sends $Y$
to the unique element $z\in\mathcal{H}$ 
satisfying $f(z,1)=0$.

The group $\mathrm{SL}_2(\ZZ)$ acts on the
set of positive definite $(2\times 2)$-matrices
via $(U,Y)\mapsto (U\transpose)^{-1} Y U^{-1} $ for $Y\in\mathrm{Mat}_2(\RR)$.
The map $\phi$ induces an isomorphism
of  $\mathrm{SL}_2(\ZZ)$-sets to $\mathcal{H}$
from
the set of positive definite 
$(2\times 2)$-matrices $Y\in\mathrm{Mat}_2(\RR)$
up to scalar multiplication.

Note that $\phi^{-1}(\mathcal{F})$
is the set of matrices $Y$ satisfying
\begin{equation}
 \label{eq:funddomy}
-y_1<2y_3\leq y_1\leq y_2,
\end{equation}
where the first two inequalities
correspond to (F1), and
the third inequality to (F2).
We say that the matrix $Y$ is \emph{$\mathrm{SL}_2$-reduced}
if it satisfies~\eqref{eq:funddomy}.

We phrase and analyze algorithm \eqref{eq:tofunddom1}
in terms of the matrices~$Y$.
Even though we will give some definitions in terms of~$Y$, all inequalities and
all steps in the algorithm will depend on $Y$ only up to scalar multiplication.

\begin{algorithm}\label{reduction}\hfil\break
 \textbf{Input:} A positive definite symmetric $(2\times 2)$-matrix $Y_0$ over~$\RR$.\\
 \textbf{Output:} $U\in \mathrm{SL}_2(\ZZ)$ and $Y=U Y_0 U\transpose $
such that $Y$ is $\mathrm{SL}_2$-reduced.\\
\noindent Start with $Y=Y_0$ and $U=1\in\mathrm{SL}_2(\ZZ)$ and iterate the
following two steps until $Y$ is $\mathrm{SL}_2$-reduced.
\begin{enumerate}
 \item Let
$$U\leftarrow \tbt{1 & 0 \\ r & 1}U\quad\mbox{and}\quad
  Y\leftarrow \tbt{1 & 0 \\ r & 1}Y\tbt{1 & r \\ 0 & 1}$$
   for $r=\lfloor -y_3/y_1+\frac{1}{2}\rfloor$.
 \item If $y_1>y_2$, then let $$U\leftarrow \tbt{0 & 1 \\ -1 & 0}U\quad\mbox{and}\quad
Y\leftarrow \tbt{0 & 1 \\ -1 & 0}Y\tbt{0 & -1 \\ 1 & 0}.$$
\end{enumerate}
Output~$U,Y$.
\end{algorithm}

We can bound the running time
in terms of the \emph{minima} of the matrix~$Y_0$.
We define the \emph{first and second minima}
$m_1(Y)$ and $m_2(Y)$
of a symmetric positive definite $(2\times 2)$-matrix
$Y$ as follows.
Let $m_1(Y)=p\transpose Yp$ be minimal among all column vectors $p\in\Z^2$
different from $0$
and let $m_2(Y)=q\transpose Yq$ be minimal among all
$q\in\Z^2$ linearly
independent of~$p$.
Note that the definition of $m_2(Y)$
is independent of the choice of~$p$.
We call $m_1(Y)$ also simply the \emph{minimum} of~$Y$.
If $Y$ is $\mathrm{SL}_2$-reduced,
then we have $$m_1(Y)=y_1,\quad m_2(Y)=y_2\quad\mbox{and}\quad
\frac{3}{4}y_1y_2\leq \det Y\leq y_1y_2,$$
so for every positive definite symmetric matrix~$Y$, we have
\begin{equation}\label{eq:minendet}
\frac{3}{4}m_1(Y)m_2(Y)\leq \det Y\leq m_1(Y)m_2(Y).
\end{equation}
As we have
$$Y^{-1}=\frac{1}{\det Y}\tbt{0 & 1 \\ -1 & 0}Y\tbt{0 & -1 \\ 1 & 0},$$
it also follows that
\begin{equation}
\label{eq:alsoreduced2}
m_i(Y^{-1})=\frac{m_i(Y)}{\det Y},\quad (i\in\{1,2\}).
\end{equation}

For any matrix~$A$, let $\abs{A}$ be the maximum of the absolute values of its entries.
\begin{lemma}\label{lem:gaussreduction}
  Algorithm \ref{reduction} is correct and takes
  $O(\log(\abs{Y_0}/m_{1}(Y_0)))$
  additions, multiplications, and divisions in~$\RR$.
  The inequalities
  $$\abs{Y}\leq \abs{Y_0}\quad\mbox{and}\quad \abs{U}\leq 2(\det{Y_0})^{-1/2}\abs{Y_0}$$
  hold for the output, and also for the values of $Y$ and $U$ throughout the execution of the algorithm.
\end{lemma}
\begin{proof}
The number
of iterations is $\leq \log(\abs{Y_0}/m_1(Y_0))/\log(3)+2$ 
by the last page of Section 7 of~\cite{lattices}.
Each has an absolutely bounded
number of $\RR$-operations.

Note that $\abs{Y}$ is decreasing throughout the algorithm.
Indeed, step 2 only swaps entries and changes signs, while step 1
decreases $\abs{y_3}$ and leaves $y_1$ and $\det Y=y_1y_2-y_3^2$ invariant,
hence also decreases~$\abs{y_2}$. This proves that we have
$\abs{Y}\leq \abs{Y_0}$ throughout the course of the algorithm.

Now let $C_0\in\mathrm{Mat}_2(\RR)$
be such that $C_0^{\vphantom{\mathrm{t}}}C_0\transpose=
Y_0^{\vphantom{\mathrm{t}}}$ holds.
Then we have $\abs{C_0}\leq \abs{Y_0}^{1/2}$
and hence $\abs{C_0^{-1}}=
\abs{\det C_0^{\vphantom{-1}}}^{-1}\abs{C_0^{\vphantom{-1}}}
\leq (\det Y_0^{\vphantom{-1}})^{-1/2}\abs{Y_0^{\vphantom{-1}}}^{1/2}$.
As we have $UC_0(UC_0)\transpose=Y$,
we also have
$\abs{UC_0}\leq \abs{Y}^{1/2}\leq \abs{Y_0}^{1/2}$.
Finally, $\abs{U}=\abs{UC_0^{\vphantom{-1}} C_0^{-1}}\leq
2\abs{UC_0^{\vphantom{-1}}}\abs{C_0^{-1}}
\leq 2(\det{Y_0^{\vphantom{-1}}})^{-1/2}\abs{Y_0^{\vphantom{-1}}}$.
\end{proof}

\subsection{The fundamental domain}
\label{sec:reducereal} \label{sec:reduceimag}\label{ssec:funddom}

For genus~$2$, the \emph{fundamental domain}
$\calF_2$ is defined
to be the
set of $Z=X+iY\in \HH_2$ for which
\begin{enumerate}
 \item[(S1)] \label{itm:realreduced2}
   the real part $X=\tbt{ x_1 & x_3 \\ x_3 & x_2}$ is reduced,
   i.e., $-\frac{1}{2}\leq x_i<\frac{1}{2}$ ($i=1,2,3$),
 \item[(S2)] \label{itm:imagreduced2}
   the imaginary part $Y$ is $\mathrm{GL}_2$-reduced,
i.e., $0\leq 2y_3\leq y_1\leq y_2$, and
\item[(S3)] \label{itm:abslarge2} $\abs{\det M^*(Z)}\geq 1$ for all $M\in \mathrm{Sp}_4(\ZZ)$,
where $M^*(Z)$ is defined
by $$M^*(Z)=CZ+D\quad\mbox{for}\quad M=\tbt{A & B \\ C & D}.$$
\end{enumerate}
Every point in $\HH_2$ is $\mathrm{Sp}_4(\ZZ)$-equivalent to a point in $\calF_2$,
and we will compute such a point with Algorithm \ref{bigreductionalgorithm} below.
This point is unique up to identifications of the boundaries of $\calF_2$.
We call points \emph{$\mathrm{Sp}_4(\ZZ)$-reduced} if they are in~$\calF_2$ .

Reduction of the real part is trivial and obtained
by $X\mapsto X+B$,
 for a unique $B\in\mathrm{Mat}_2(\ZZ)$.
Here $X\mapsto X+B$
corresponds to the action of
$$\tbt{1 & B \\ 0 & 1}\in \mathrm{Sp}_4(\ZZ).$$

Reduction of the imaginary part is
$\mathrm{SL}_2$-reduction 
as in Algorithm \ref{reduction}, but with the extra condition $y_3\geq 0$,
which is obtained by applying the $\mathrm{GL}_2(\ZZ)$-matrix
$\mathrm{diag}(1,-1)$.
It follows that $U Y U\transpose $ is $\mathrm{GL}_2$-reduced for some
$U\in\mathrm{GL}_2(\ZZ)$,
and to reduce the imaginary part of~$Z$, we replace $Z$ by
\begin{equation}\label{eq:symplecticgauss}
U Z U\transpose =\tbt{U & 0 \\ 0 & (U\transpose )^{-1}}(Z).
\end{equation}

Condition (S3) has a finite formulation.
Let $\mathfrak{G}$ consist of the $38$ matrices
\begin{align*}
\setlength{\arraycolsep}{2pt}
\left(\begin{array}{rrrr}
0 & 0 &-1 & \phantom{-}0\\
0 & \phantom{-}1 & 0 & 0\\
 1 & 0 &e_{\rlap{{\scriptsize 1}}}&0\\
0 & 0 & 0 & 1
\end{array}\right), & &
\setlength{\arraycolsep}{2pt}
\left(\begin{array}{rrrr}
1 & 0 & 0 & 0\\
0 & 0 & 0 &-1\\
0&0&1 & 0\\
0 &\phantom{-}1 &\phantom{-}0 & e_{\rlap{{\scriptsize 1}}}
\end{array}\right),& &
\setlength{\arraycolsep}{2pt}
\left(\begin{array}{rrrr}
0 & 0 &-1 & 0\\
0 & 1 & 0 &\phantom{-}0\\
1 &-1 & d & 0\\
0 & 0 & 1 & 1
\end{array}\right), & &
\setlength{\arraycolsep}{2pt}
\left(\begin{array}{rrrr}
0 & \phantom{-}0 &-1 & 0\\
0 & 0 & 0 &-1\\
1 & 0 & e_{\rlap{{\scriptsize 1}}} & e_{\rlap{{\scriptsize 3}}}\\
0 & 1 & e_{\rlap{{\scriptsize 3}}} & e_{\rlap{{\scriptsize 2}}}
\end{array}\right),
\end{align*}
in $\mathrm{Sp}_4(\ZZ)$,
where $d$ ranges over $\{0,\pm 1,\pm 2\}$ and each $e_i$ over $\{0,\pm 1\}$.
Gottschling \cite{gottschling} proved that,
under conditions (S1) and (S2), condition (S3) is equivalent to the condition
\begin{enumerate}
\item[(G)]
$\abs{\det M^*(Z)}\geq 1\quad\mbox{for all $M\in\mathfrak{G}$}.$
\end{enumerate}
Actually, Gottschling
went even further and gave a subset
of $19$ elements of $\mathfrak{G}$ of which he proved that it is
minimal such that (G) is equivalent to (S3), assuming (S1) and (S2).

For our purposes of bounding and computing the values
of Igusa invariants, it suffices to consider the set
$\mathcal{B}\subset \HH_2$, given by (S1), (S2), and
\begin{enumerate}
\item[(B)]\quad \quad  $y_1\geq\sqrt{3/4}$.
\end{enumerate}
Note that the set $\mathcal{B}$ contains~$\calF_2$.
Indeed, condition (B) follows immediately from (S1) and
$\abs{z_1}=\abs{\det(N_0^*(Z))}\geq 1$,
where $N_0$ is the 
first matrix in our defintion of $\mathfrak{G}$
(with $e_1=0$).

\subsection{The reduction algorithm}\label{ssec:reduction}

We move $Z\in\HH_2$ into $\calF_2$
as follows.

\begin{algorithm}\label{bigreductionalgorithm}\label{alg:funddom}\hfil\break
  \textbf{Input:} $Z_0\in \HH_2$.\\
  \textbf{Output:} $Z$ in $\calF_2$ and a matrix
  $M\in\mathrm{Sp}_4(\ZZ)$ satisfying $Z=M(Z_0)$.
  
 \noindent Start with $Z=Z_0$ and iterate the
  following $3$ steps.
  During the course of the algorithm, keep track of $M\in\mathrm{Sp}_4(\ZZ)$ such that
  $Z=M(Z_0)$, as we did with $U$ in Algorithm~\ref{reduction}.
 \begin{enumerate}
  \item \label{step:realred} Reduce the imaginary part as explained in Section~\ref{sec:reducereal}.
  \item \label{step:imagred} Reduce the real part as explained in Section~\ref{sec:reduceimag}.
  \item \label{step:applyn} Apply $N$ to $Z$ for $N\in\mathfrak{G}$ with $\abs{\det N^*(Z)}<1$ minimal,
        if such an $N$ exists. Otherwise, return~$Z$ and~$M$.
 \end{enumerate}
\end{algorithm}
The algorithm that moves $Z\in\HH_2$ into $\mathcal{B}$ is exactly the same,
but with
$\mathfrak{G}$ replaced
by~$\{N_0\}$. We will give an analysis
of the running time and output of
Algorithm~\ref{alg:funddom} below.
The only property of the subset $\mathfrak{G}\subset\mathrm{Sp}_4(\ZZ)$
that this analysis
uses is that it is finite and contains~$N_0$, hence 
the analysis is equally valid for the modification that
moves points into~$\mathcal{B}$.

\subsection{The number of iterations}

We will bound the number of iterations by showing that
$\det Y$ is increasing and bounded in terms of $Y_0$,
that 
every step with $\abs{y_1}<\frac{1}{2}$
leads to a doubling of~$\det Y$,
and that
we have an absolutely bounded number of steps
with $\abs{y_1}\geq \frac{1}{2}$.

\begin{lemma}\label{lem:klingen1punt6det}
For any point $Z\in\HH_2$ and any matrix $M\in\mathrm{Sp}_4(\ZZ)$, we have
$$\det\Im M(Z)=\frac{\det \Im Z}{\abs{\det M^*(Z)}^{2}}.$$
\end{lemma}
\begin{proof}
In \cite[Proof of Proposition 1.1]{klingen} it is computed that
\begin{equation}\label{klingen1punt6}
\Im M(Z)=(M^*(Z)^{-1} )\transpose(\Im Z)M^*(\overline{Z})^{-1}.
\end{equation}
Taking determinants on both sides proves the result.
\end{proof}
Steps 1 and 2 of Algorithm \ref{alg:funddom}
do not change~$\det Y$,
and Lemma \ref{lem:klingen1punt6det} shows that
step~3 increases $\det Y$,
so $\det Y$ is increasing throughout the algorithm.

\begin{lemma}\label{lem:factorof22}
  At every iteration of step 3 of Algorithm \ref{bigreductionalgorithm}
  in which we have $y_1<\frac{1}{2}$, the value of $\det Y$ increases by a factor of at least~$2$.
\end{lemma}
\begin{proof}
 If $y_1<\frac{1}{2}$, then
for
the element $N_0\in\mathfrak{G}$
(defined in the line above Section~\ref{ssec:reduction}),
we have
$\abs{\det N_0^*(Z)}^2=\abs{z_1}^2=\abs{x_1}^2+\abs{y_1}^2\leq\frac{1}{2}$,
so by Lemma \ref{lem:klingen1punt6det}, the value of $\det Y$ increases by a factor~$\geq 2$.
\end{proof}

\begin{lemma}\label{lem:factorof21}
  There is an absolute upper bound~$c$, independent of the input $Z_0$,
  on the number of iterations of Algorithm \ref{bigreductionalgorithm}
  in which $Z$ satisfies $y_1\geq \frac{1}{2}$ at the beginning of step~3.
\end{lemma}
\begin{proof}
 Let $\mathcal{C}$ be the set of points in $\HH_2$ that satisfy
 (S1), (S2) and $y_1\geq \frac{1}{2}$.
 At the beginning of step~3, both (S1) and (S2) hold, so we need to bound
 the number of iterations for which $Z$
 is in $\mathcal{C}$ at the beginning of step 3.
 Suppose that such an iteration exists,
 and denote the value of $Z$ at the beginning
 of step 3 of that iteration by~$Z'$.
 As $\det Y$ increases during the algorithm, each iteration has a different value of~$Z$,
 so it suffices to bound the number of
 $Z\in \mathrm{Sp}_4(\ZZ)(Z')\cap \mathcal{C}$.
  By \cite[Theorem 3.1]{klingen}, the set
$$\mathfrak{C}=\{M\in\mathrm{Sp}_4(\ZZ):\mathcal{C}\cap M(\mathcal{C})\not=\emptyset\}$$
  is finite.
  As $\mathfrak{C}$ surjects onto $\mathrm{Sp}_4(\ZZ)(Z')\cap \mathcal{C}$ via $M\mapsto M(Z')$, we get
  the absolute upper bound $\#\mathfrak{C}$
  on the number of iterations with $Z\in \mathcal{C}$.
\end{proof}

For bounding the number of iterations, we now only
need to bound $\det Y$ from above in terms of the input~$Y_0$.
For this, we use the following result,
which will also help us bound the sizes of the numbers
encountered.
\begin{lemma}\label{lem:boundm2m}
For any point $Z=X+iY\in\HH_2$ and any matrix $M\in\mathrm{Sp}_4(\ZZ)$, we have
  $$m_2(\Im M(Z))\leq \frac{4}{3}\max\{ m_1( Y)^{-1}, m_2(Y)\}.$$
\end{lemma}
\begin{proof}
   We imitate part of the proof of \cite[Lemma 3.1]{klingen}.
   If we replace $M$ by $$\tbt{(U\transpose )^{-1} & 0 \\ 0 & U}M$$
for $U\in\mathrm{GL}_2(\ZZ)$, then the matrix
$(\Im M(Z))^{-1}$ gets replaced by the matrix
$U(\Im M(Z))^{-1}U\transpose $,
so we can assume without loss of generality that
$(\Im M(Z))^{-1}$ is reduced.
  By (\ref{klingen1punt6}), we have
\begin{align}
(\Im M(Z))^{-1}  &=  (CX-iCY+D)Y^{-1}(CX+iCY+D)\transpose \nonumber \\
                 &=  (CX+D)Y^{-1}(XC\transpose +D\transpose )+CYC\transpose ,\label{eq:invimm}\\
\mbox{where}\quad
              M  &=  \tbt{A & B \\ C & D}.\nonumber
\end{align}
As the left hand side of \eqref{eq:invimm} is reduced,
its minimum
$m_1$ is its upper left entry.
  Denote the third row of $M$
by $(c_1,c_2,d_1,d_2)$
  and let $c=(c_1,c_2)$, $d=(d_1,d_2)\in\Z^2$.
We compute that the upper left entry of \eqref{eq:invimm}
is $m_1((\Im M(Z))^{-1})=(cX+d)Y^{-1} (Xc\transpose +d\transpose )+c Y c\transpose $.

The matrix $M$ is invertible, so if $c$ is zero, then $d$ is non-zero.
As both $Y^{-1}$ and $Y$ are positive definite, this implies
$$m_1((\Im M(Z))^{-1})\geq \min\{m_1 ( Y), m_1 (Y^{-1})\}.$$
By (\ref{eq:minendet}) and (\ref{eq:alsoreduced2}), we get
\begin{align*}
m_2(\Im M(Z))  &\leq  \frac{4\det \Im M(Z)}{3 m_1(\Im M(Z))}=
                            \frac{4}{3 m_1((\Im M(Z))^{-1})}                     \\
               &\leq  \frac{4}{3}\max\{\frac{1}{m_1(Y)}, \frac{\det Y}{m_1(Y)}\}
               \leq  \frac{4}{3}\max\{m_1(Y)^{-1}, m_2(Y)\},
\end{align*}
which proves the result.
\end{proof}

We can now bound the number of iterations.
For any matrix $Z=X+iY\in \HH_2$, let $t(Z)=\max\{\log (m_1(Y)^{-1}),\log (m_2(Y)), 1\}$.
\begin{proposition}\label{prop:kiterations}
  The
  number of iterations of Algorithm \ref{bigreductionalgorithm} is at most
  $O(t(Z_0))$ for every input~$Z_0$.
\end{proposition}
\begin{proof}
  Let $c$ be the constant of Lemma \ref{lem:factorof21},
  let $Z_0$ be the input of Algorithm \ref{bigreductionalgorithm}
  and let $Z$ be what it is after $k$ iterations.
By Lemmas \ref{lem:factorof22}--\ref{lem:boundm2m}, we have
  \begin{equation*}
    2^{k-c} \det Y_0\leq \det Y\leq m_2(Y)^2\leq (\frac{4}{3})^2\max\{m_1(Y_0)^{-2},m_2(Y_0)^2\},
  \end{equation*}
  hence \eqref{eq:minendet} implies
\[2^{k-c}\leq (\frac{4}{3})^3\max\{m_1(Y_0)^{-3}m_2(Y_0)^{-1},m_1(Y_0)^{-1}m_2(Y_0)\}.\qedhere\]
\end{proof}

To avoid a laborious error analysis, all computations are performed
inside some number field $L\subset \C$ of absolutely bounded degree.
Indeed, for an abelian surface $A$ with CM by $\OK$,
any period matrix $Z\in\HH_2$ that represents~$A$
is in $\mathrm{Mat}_2(L)$, where $L$ is the normal closure of~$K$,
which has degree at most~$8$.
For a running time analysis,
we need to bound the \emph{height} of the numbers involved.
Such height bounds are also used for lower bounds on the off-diagonal part of the
output~$Z$, which we will need in Section~\ref{sec:theta}.

The height $h(x)$ of an element $x\in L^*$ is defined as follows.
Let $S$ be the set of absolute values of $L$
that extend either the standard archimedean
absolute value of $\Q$ or one of the non-archimedean
absolute values $\abs{x}=p^{-\mathrm{ord}_p(x)}$.
For each $v\in S$, let $\mathrm{deg}(v)=[L_v:\Q_v]$ be the degree of the
completion $L_v$ of $L$ at~$v$.
Then $$h(x)=\sum_{v}\mathrm{deg}(v)\max\{\log \abs{x}_v,1\}.$$
We denote the maximum of the heights of all non-zero entries of
a matrix $Z\in\HH_2$ by~$h(Z)$.

\subsection{The size of the numbers}

Next, we give bounds on~$\abs{M}$.
This will provide us with a bound
on the height of the entries of~$Z$.
Indeed, if $Z=M(Z_0)$, then 
$h(Z)\leq 16(\log\abs{M}+h(Z_0)+\log 4)$.

\begin{lemma}\label{lem:bincreases}
 There exists an absolute constant $c>0$ such that
 the following holds.
 The value of $\log \abs{M}$ is at most $c\max\{\log \abs{Z_0},t(Z_0)\}$
 during the first iteration of Algorithm \ref{bigreductionalgorithm}
 and, in each iteration, increases by at most
 $c t(Z_0)$, where $t$ is as above Proposition~\ref{prop:kiterations}.
\end{lemma}
\begin{proof}
For step~1, it follows from equation \eqref{eq:symplecticgauss}
and Lemma~\ref{lem:gaussreduction} that
$\log\abs{M}$ increases by at most
$\log\abs{Y}-\frac{1}{2}\log\det{Y}+\log 4$.
As said below Lemma~\ref{lem:klingen1punt6det},
the determinant of $Y$ decreases thoughout the algorithm,
so we conclude that $\log\abs{M}$ increases by at most
$\log\abs{Z}+t(Z_0)+\log 8$ in step~1.
We still have to bound $\log\abs{Z}$ appropriately.

The value of $\log\abs{M}$ increases
by at most $\log(2+2\abs{Z})$ in step~2
and at most $\log 4$ in step~3.

Next, we bound $\log \abs{Z}$
at the beginning of steps 1 and~2.
Note that $\log \abs{Y}$ decreases during step~1, while
$\log \abs{X}$ increases by at most
$2(\log\abs{Z}+t(Z_0)+\log 8)$.
At the beginning of the first iteration, we have $Z=Z_0$,
proving the bound $c\max\{ \log\abs{Z_0}, t(Z_0)\}$
in the lemma.
It now suffices
to prove $\log\abs{Z} = O(t(Z_0))$ at the beginning of step~1
for all other iterations, i.e., at the end of step~3
for all iterations.

At the beginning of step~3, we have $\abs{X}\leq \frac{1}{2}$,
and $Y$ is reduced. Lemma~\ref{lem:boundm2m}
therefore gives $\abs{Y}\leq 4e^{t(Z_0)}/3$,
which implies $\log\abs{Z}\leq 3 t(Z_0)$.
During step~3, the matrix $Z$ gets replaced with
\begin{align*}
N(Z)&=(AZ+B)(CZ+D)^{-1}
\quad\quad\mbox{for}\quad\quad
N=\tbt{A & B \\ C & D}\in\mathfrak{G},\quad\quad\mbox{so}\\
\abs{N(Z)}& \leq\abs{\det(CZ+D)}^{-1}2(2\abs{A}\abs{Z}+\abs{B})
(2\abs{C}\abs{Z}+\abs{D})\\
&\leq\abs{\det(CZ+D)}^{-1}2(2\abs{Z}+1)^2\abs{N}^2.
\end{align*}
We already have~$\log \abs{Z}\leq 3t(Z_0)$ and $\abs{N}\leq 2$,
so it suffices to prove $\log(\abs{\det(CZ+D)}^{-1})= O(t(Z_0))$.
Lemma~\ref{lem:klingen1punt6det}
gives \[\abs{\det(CZ+D)}^{-2}=(\det \Im N(Z))(\det \Im(Z))^{-1}.\]
Let $M'\in\mathrm{Sp}_{4}(\Z)$ satisfy $Z=M'(Z_0)$ and let $M=NM'$, then
\eqref{eq:minendet} and Lemma~\ref{lem:boundm2m} tell us 
$$\det\mathrm{Im}(N(Z))\leq (\frac{4}{3}\max\{m_1(Y_0)^{-1},m_2(Y_0)\})^2.$$
Applying the fact that the determinant of $\Im(Z)$ increases throughout
the algorithm, we get $(\det\mathrm{Im}(Z))^{-1}\leq (\det \Im(Z_0))^{-1}
\leq \frac{4}{3} m_1(Y_0)^{-1} m_2(Y_0)^{-1}$,
hence
\begin{align*}
\log (\abs{\det(CZ+D)}^{-1}) &\leq (3/2) \log (4/3) + 2t(Z_0).
 \end{align*}
Therefore, for $Z$ and $N$ as in step~3,
we have $\log\abs{N(Z)}=O(t(Z_0))$,
hence $O(t(Z_0))$ is an upper bound for $\log\abs{Z}$
at the beginning of step 1 for every iteration but the first.
\end{proof}

\subsection{The running time}

\begin{theorem}\label{thm:movingup}
 Let $L\subset\C$ be a number field.
 Algorithm \ref{bigreductionalgorithm}, on input $Z_0\in\mathrm{Mat}_2(L)\cap \HH_2$,
 returns an $\mathrm{Sp}_4(\ZZ)$-equivalent
 matrix $Z\in\calF_2$.
 The running time is
 $\widetilde{O}(h(Z_0)^4+1)$.
Moreover, the
output $Z$ satisfies
$h(Z) = O(h(Z_0)^2+1)$.
\end{theorem}
\begin{proof}
By Proposition \ref{prop:kiterations} and Lemma \ref{lem:bincreases},
the value of $\log \abs{M}$
is bounded by $O(\log \abs{Z_0}+t(Z_0)^2+1)$ throughout the algorithm,
so the height of every entry of $Z$ is bounded by $O(t(Z_0)^2)+O(h(Z_0))$.
This implies that
each basic arithmetic operation in the algorithm takes time at most $\widetilde{O}(t(Z_0)^2)+\widetilde{O}(h(Z_0))$.
By Lemma~\ref{lem:gaussreduction}, the first iteration takes
$O(\log \abs{Z_0})+O(t(Z_0))$ such operations, 
and all other $O(t(Z_0))$ iterations take $O(t(Z_0))$ operations,
so there are $O(\log \abs{Z_0})+O(t(Z_0)^2)$ arithmetic operations,
yielding a total running time for the algorithm of  $\widetilde{O}(t(Z_0)^4)+\widetilde{O}(h(Z_0)\log \abs{Z_0})$. The bounds of the lemma
follow once we prove $t(Z_0) = O(h(Z_0)+1)$.

Note $\log m_2(Z_0) \leq \log |Z_0|\leq h(Z_0)$
and $\log (m_1(Z_0)^{-1}) \leq \log m_2(Z_0) +\log (\det(Y_0))^{-1})
\leq h(Z_0) + h(\det(Y_0)^{-1}) \leq h(Z_0) + h(\det(Y_0)) = O(h(Z_0)+1)$.
\end{proof}
In Section \ref{sec:theta}, we bound the Igusa invariants in terms of the entries
of the period matrix~$Z$.
One of the bounds that we need in that section
is a lower bound on the absolute value of the off-diagonal
entry $z_3$ of~$Z$.
It is supplied by the following corollary.
\begin{corollary}\label{cor:movingup}
  Let $Z_0\in\mathrm{Mat}_2(L)\cap \HH_2$ be the input of Algorithm \ref{bigreductionalgorithm}
  and let $z_3$ be the off-diagonal entry of the output.
  Then we have either $z_3=0$ or
  $-\log \abs{z_3}\leq 
  O(h(Z_0)^2+1)$.
\end{corollary}
\begin{proof}
  The field $L$ is a subfield of $\C$,
  which gives us a standard absolute value~$v$.
  If $z_3$ is non-zero, then the product formula tells us that we have
  $-\log\abs{z_3}=-\log\abs{z_3}_v=\sum_{w\not=v}\log\abs{z_3}_w\leq h(z_3)
  =O(h(Z_0)^2+1)$
\end{proof}

\section{Theta constants}

\label{sec:theta}

To compute the absolute Igusa invariants corresponding to a point
$Z\in\HH_2$, we use a formula of Igusa that expresses
them in terms of \emph{theta constants}.
For $z\in \C$, let $e(z)=e^{2\pi i z}$.
We call an element $c\in\{0,\frac{1}{2}\}^4$ a
\emph{theta characteristic}
and write $c=(c_1,c_2,c_3,c_4)$, $c'=(c_1,c_2)$ and $c''=(c_3,c_4)$.
We define the
\emph{theta constant of characteristic $c$} to be the function
$\theta[c]:\HH_2\rightarrow\C$ given by\label{defoftheta}
\[ 
\theta
[ c](Z)= \sum_{n\in \ZZ^2} e(\frac{1}{2}(n+c') Z (n+c')\transpose +(n+c') {c''}\transpose ),\]
and following Dupont \cite{dupont}, we use the short-hand notation
$$
\theta_{16c_2+8c_1+4c_4+2c_3}=\theta [c].
$$
We call a theta characteristic
--- and the corresponding theta constant ---
even or odd depending on whether
$4c'{c''}\transpose $ is even or odd.
The odd theta constants are zero by the anti-symmetry in the definition,
and there are exactly $10$ even theta constants
$\theta_{0},\theta_{1},\theta_{2},\theta_{3},\theta_{4},\theta_{6},\theta_{8},\theta_{9},\theta_{12}$ and $\theta_{15}$.

\subsection{Igusa invariants in terms of theta constants}

Let $T$ be the set of even theta characteristics and define
$$S=\{ C\subset T \mid \# C =4, \sum_{c\in C} c\in\Z^4\}.$$
Then $S$ consists of $15$ subsets of $T$ called \emph{G\"opel quadruples},
each consisting
of $4$ even theta characteristics.
We call a set $\{b,c,d\}\subset T$ of
three distinct even theta characteristics
\emph{syzygous}
if it is a subset of a G\"opel quadruple,
so there are $60$ syzygous triples.
Define
\begin{align}\label{eq:defofh}
  h_4    &= \sum_{c\in T}\theta[c]^8,& 
  h_{6}  &= \sum_{\substack{b, c, d\in T \\ \mathrm{syzygous}}} \pm (\theta[b] \theta[c] \theta[d])^4\\
  h_{10} &= \prod_{c\in T}\theta[c]^2,&
  h_{12} &= \sum_{C\in S}\prod_{c\in T\setminus C} \theta[c]^4,\nonumber
\end{align}
where we explain the signs in $h_6$ below.
Each $h_{k}$ is a sum of $t_k$ monomials of degree $2k$ in the $10$ even theta constants,
where $t_4=10$, $t_6=60$, $t_{10}=1$, and $t_{12}=15$.
The signs in $h_6$ are defined uniquely by the facts
that $h_6$
is a modular form for $\mathrm{Sp}_4(\ZZ)$ and that
the coefficient of $\theta_0^4\theta_1^4\theta_2^4$
is~$+1$.
More explicitly, we give $h_6$ in Figure~\ref{fig:h6}.

\newcommand{\tinystar}{{\scriptscriptstyle *}}
\begin{figure}
 \[\begin{array}{l}
    t0\tinystar t1\tinystar t2 + t0\tinystar t1\tinystar t3 + t0\tinystar t2\tinystar t3 + t1\tinystar t2\tinystar t3
- t0\tinystar t2\tinystar t4 + t1\tinystar t3\tinystar t4 - t0\tinystar t2\tinystar t6\\
 + t1\tinystar t3\tinystar t6 -
t0\tinystar t4\tinystar t6 - t1\tinystar t4\tinystar t6 - t2\tinystar t4\tinystar t6 - t3\tinystar t4\tinystar t6 -
t0\tinystar t1\tinystar t8 + t2\tinystar t3\tinystar t8\\ + t0\tinystar t4\tinystar t8 + t3\tinystar t4\tinystar t8 -
t1\tinystar t6\tinystar t8 - t2\tinystar t6\tinystar t8 - t0\tinystar t1\tinystar t9 + t2\tinystar t3\tinystar t9 -
t1\tinystar t4\tinystar t9\\ - t2\tinystar t4\tinystar t9 + t0\tinystar t6\tinystar t9 + t3\tinystar t6\tinystar t9 -
t0\tinystar t8\tinystar t9 - t1\tinystar t8\tinystar t9 - t2\tinystar t8\tinystar t9 - t3\tinystar t8\tinystar t9\\ +
t1\tinystar t2\tinystar t12 - t0\tinystar t3\tinystar t12 + t0\tinystar t4\tinystar t12 + t1\tinystar t4\tinystar t12
- t2\tinystar t6\tinystar t12 - t3\tinystar t6\tinystar t12\\ + t0\tinystar t8\tinystar t12 +
t2\tinystar t8\tinystar t12 + t4\tinystar t8\tinystar t12 + t6\tinystar t8\tinystar t12 - t1\tinystar t9\tinystar t12
- t3\tinystar t9\tinystar t12\\ + t4\tinystar t9\tinystar t12 + t6\tinystar t9\tinystar t12 +
t1\tinystar t2\tinystar t15 - t0\tinystar t3\tinystar t15 - t2\tinystar t4\tinystar t15 - t3\tinystar t4\tinystar t15\\
+ t0\tinystar t6\tinystar t15 + t1\tinystar t6\tinystar t15 - t1\tinystar t8\tinystar t15 -
t3\tinystar t8\tinystar t15 + t4\tinystar t8\tinystar t15 + t6\tinystar t8\tinystar t15\\ + t0\tinystar t9\tinystar t15
+ t2\tinystar t9\tinystar t15 + t4\tinystar t9\tinystar t15 + t6\tinystar t9\tinystar t15 -
t0\tinystar t12\tinystar t15 - t1\tinystar t12\tinystar t15\\ - t2\tinystar t12\tinystar t15 -
t3\tinystar t12\tinystar t15
   \end{array}\]
\caption{An explicitly written out version of~$h_6$
(see~\eqref{eq:defofh}). We write $tj$ instead of $\theta_j^4$
for ease of copying with a computer.}\label{fig:h6}
\end{figure}

\begin{remark} \label{rem:igusaineis}
  Another way of defining $h_k$ is
  by letting $\psi_k$ be the Eistenstein series
  of weight $k$ on $\mathcal{H}_2$ and setting
  $h_4=2^2\psi_4$, $h_6=2^2\psi_6$,
\begin{align*}
 h_{10}& =-2^{14}\chi_{10} & \mbox{for}\quad \chi_{10}
  &= -43867(2^{12}3^{5}5^{2}7\cdot 53)^{-1}(\psi_4\psi_6-\psi_{10}),
  \quad \mbox{and}\\
 h_{12}&=2^{17}3\chi_{12} & \mbox{for}\quad \chi_{12}
  &= 131\cdot 593( 2^{13}3^{7}5^{3}7^{2}337)^{-1}(3^2 7^2\psi_4^3
     +2\cdot 5^3\psi_6^2-691\psi_{12}).
\end{align*} 
 See also Igusa \cite[p.~189]{igusa-i}
 and~\cite[p.~848]{igusa-modformsprojective}.
\end{remark}

\begin{lemma}
\label{lem:igusaintheta}
  Let $Z$ be a point in $\HH_2$. If $h_{10}(Z)$ is non-zero, then
  the principally polarized abelian variety corresponding to $Z$
  is the Jacobian of a curve $C/\C$ of genus~$2$ with
  invariants
\begin{align*}
   \homigu{2}{}(C)       &= h_{12}(Z)/h_{10}(Z), &
   \homigu{4}{}(C)       &= h_4(Z),\\
   \homigu{6}{\prime}(C) &= h_{6}(Z), &
   \homigu{10}{}(C)      &= h_{10}(Z).
\end{align*}
\end{lemma}
\begin{proof}
This is the result on page 848 of
Igusa~\cite{igusa-modformsprojective}.
\end{proof}
\begin{corollary}
\label{cor:igusaintermsofmodforms}\label{cor:sufficestoboundtheta}
With $Z$ and $C$ as in Lemma~\ref{lem:igusaintheta},
we have $i_1(Z) = h_{4}^{\vphantom{2}}h_{6}^{\vphantom{2}}h_{10}^{-1}$,
$i_2(Z) = h_{4}^2h_{12}^{\vphantom{2}}h_{10}^{-2}$,
$i_3(Z) = h_{4}^5 h_{10}^{-2}$.
More generally, each element of the ring $A=\Q[\homigu{2}{},\homigu{4}{},\homigu{6}{\prime},\homigu{10}{-1}]$
can be expressed as a polynomial in the theta constants divided by a power
of the product of all even theta constants.\qed
\end{corollary}
\begin{remark}
Thomae's formula (\cite[Thm.~IIIa.8.1]{tataii}, 
\cite{Thomae}) gives a defining equation in terms of theta constants
 for a curve $C$
with $J(C)$ corresponding to a given~$Z$.
Formulas of the form of Lemma \ref{lem:igusaintheta}
can be derived by writing out the definition
of $\homigu{k}{}$ using Thomae's formula
and standard identities between the theta constants.
This was done by Bolza~\cite{bolza}
for older invariants, and later
by Spallek~\cite{spallek}.
Spallek did not give~$h_{6}$, but instead gave an
explicitly written out version of $h_4$, $h_{10}$, $h_{12}$, and
  \[h_{16} = \sum_{\substack{C\in S \\ d\in C}}
  \theta[d]^8\prod_{c\in T\setminus C}\theta[c]^4,\]
together with the formulas for $I_2$, $I_4$, $I_{10}$
of Lemma \ref{lem:igusaintheta} and the formula
\[\homigu{6}{}(C)=h_{16}(Z)/h_{10}(Z).\]
The same big formulas later appeared in 
\cite{hehcc18,weng},
and with a simplification in~\cite{dupont}.
We choose to use only $h_4$, $h_6$, $h_{10}$ and $h_{12}$,
not the higher-weight~$h_{16}$, and to use Igusa's formulas
as they are more compact.
\end{remark}

\subsection{Bounds on the theta constants}

\label{ssec:boundtheta}

To bound the height of Igusa class polynomials,
we have to bound $|i_n(Z)|$ from above,
where $Z$ is a period matrix in the fundamental domain from Section~\ref{sec:up}.
We will see that the theta constants,
and hence the numerators in the expressions of Corollary~\ref{cor:igusaintermsofmodforms},
are bounded from above by a constant,
so that the main task is to bound $h_{10}(Z)=\prod \theta[c](Z)^2$ away from zero.
Bounding $h_{10}(Z)$ away from zero is also crucial
for controlling the precision loss in the division.

For $Z\in \HH_2$, denote the real part of $Z$
by $X$ and the imaginary part by~$Y$,
write $Z$ as
$$Z=\tbt{z_1 & z_3 \\ z_3 & z_2},$$
and let $x_j$ be the real part of $z_j$ and $y_j$ the imaginary part
for $j=1,2,3$.
Recall that $\mathcal{B}\subset \HH_2$ is given by
\begin{enumerate}
 \item[(S1)] $X$ is reduced, i.e., $-1/2\leq x_i< 1/2$ for $i=1,2,3$,
 \item[(S2)] $Y$ is reduced, i.e., $0\leq 2y_3\leq y_1\leq y_2$, and
 \item[(B)] $y_1\geq \sqrt{3/4}$.
\end{enumerate}

\begin{proposition}\label{prop:thetabounds}
  For every $Z\in \mathcal{B}$, we have
   \begin{align*}
   \abs{\theta_j(Z)-1}& < 0.405 & j\in\{0,1,2,3\}\\
   \abs{\frac{\theta_j(Z)}{2e(\frac{1}{8}z_1)}-1}& < 0.348 & j\in\{4,6\}\\
   \abs{\frac{\theta_j(Z)}{2e(\frac{1}{8}z_2)}-1}& < 0.348 & j\in \{8,9\}\quad\mbox{and}\\
   \abs{\frac{\theta_j(Z)}{2((-1)^j+ e(\frac{1}{2} z_3))e(\frac{1}{8}(z_1+z_2-2z_3))}-1}& < 0.438 & j\in\{12,15\}.
\end{align*}
\end{proposition}
\begin{proof}
  The proof of Proposition 9.2 of Klingen \cite{klingen} gives
  infinite series as upper bounds for the left hand sides.
  A numerical inspection shows that the limits of these series
  are less than $0.553$, $0.623$, $0.623$ and $0.438$.
  Klingen's bounds can be improved by estimating
  more terms of the theta constants
  individually and thus getting a smaller error term.
  This has been done
  in Propositions 6.1 through 6.3 of Dupont~\cite{dupont}, 
  improving the first three bounds to $0.405$, $2 \abs{e(z_1/4)}\leq 0.514$ and
  $2\abs{e(z_2/4)}\leq 0.514$.
  The proof of \cite[Proposition 6.2]{dupont} shows that
  for the second and third bound, we can also take $0.348$.
\end{proof}
%
\begin{corollary}\label{cor:thm2upperandlower}
 For every $Z\in \mathcal{B}$,
  we have
$$
\begin{array}{rccclr}
  0.59 &<&\abs{\theta_j(Z)} &<& 1.41,  & (j\in\{0,1,2,3\})\\
   1.3\ \exp(-\frac{\pi}{4}y_1) &<&\abs{\theta_j(Z)}&<&1.37,  & (j\in\{4,6\})\\
   1.3\ \exp(-\frac{\pi}{4}y_2) &<&\abs{\theta_j(Z)}&<&1.37,  & (j\in\{ 8, 9\})\\
  1.05\ \exp(-\frac{\pi}{4}(y_1+y_2-2y_3))
      &<& \abs{\theta_{12}(Z)}&<& 1.56, & \mbox{and}\\
  1.12\ \exp(-\frac{\pi}{4}(y_1+y_2-2y_3))\nu &<& \abs{\theta_{15}(Z)}&<& 1.56,
\end{array}
$$
where
$\nu=
\mathrm{min}\{\frac{1}{4},\abs{z_3}\}$.
\end{corollary}
\begin{proof}
The upper bounds follow immediately from
(S2), (B), and
Proposition~\ref{prop:thetabounds}.
  The lower bounds follow from Proposition \ref{prop:thetabounds}
  if we use $|1-e(z_3/2)|\geq \nu$
  and
  the bounds
  $$\abs{1+e(z_3/2)}>1,\quad \exp(-\frac{\pi}{4} y_i)\geq 0.506\quad (i\in\{1,2\})\quad\mbox{and}$$
 \[
\exp\left(-\frac{\pi}{4}(y_1+y_2-2\abs{y_3})\right)
  >\exp\left(-\frac{\pi}{2}y_2\right)\geq 0.256.\qedhere
\]
\end{proof}
\begin{corollary}\label{cor:thm2numanddenom}
For every $Z\in\mathcal{B}$, we have
\[\log_2\abs{h_4(Z)} < 8, \quad \log_2\abs{h_6(Z)} < 13,
\quad \log_2\abs{h_{10}(Z)}  < 11, \quad\log_2\abs{h_{12}(Z)} < 17,\quad\mbox{and}\]
$$-\log_2\abs{h_{10}(Z)}  < 
\pi(y_1+y_2-y_3) +3+\max\{2,-\log_2\abs{z_3}\}.$$
\end{corollary}
\begin{proof}
This follows from the upper and lower bounds in Corollary\ref{cor:thm2upperandlower}.
\end{proof}
\begin{theorem}\label{thm:thm2}
 For every $Z\in \mathcal{B}$ and $n\in\{1,2,3\}$,
  we have
$$
  \log_2\abs{i_n(Z)}  <  2\pi( y_1+y_2-y_3)+64
  +2\max\{2,-\log_2\abs{z_3}\}.
$$
\end{theorem}
\begin{proof}
This follows from Corollary~\ref{cor:thm2numanddenom}
and the formulas in Corollary~\ref{cor:igusaintermsofmodforms}.
\end{proof}
\begin{remark}\label{rem:constructiveweil}
  Lemma
  \ref{lem:igusaintheta}, together with
  Corollary~\ref{cor:thm2numanddenom}, 
  gives a constructive version of (Weil's)
  Theorem~\ref{weilsthm}.
  Indeed, if $z_3=0$, then the principally polarized
  abelian surface $A(Z)$ corresponding to $Z$
  is the product of the polarized elliptic
  curves $\C/(z_1\Z+\ZZ)$ and $\C/(z_2\Z+\ZZ)$, while
  if $z_3\not=0$, then
  Corollary~\ref{cor:thm2numanddenom} implies
  $h_{10}(Z)\not=0$, so $A(Z)$ is the Jacobian
  of the curve of
  Lemma~\ref{lem:igusaintheta}.
\end{remark}

\subsection{Evaluating theta constants and Igusa invariants}

\label{ssec:computetheta}

We use the naive way of evaluating theta constants.
That is, we simply sum all terms in the definition of~$\theta$
with~$|n_i|\leq R$ for
$$ R = \lceil (0.51 s + 2.55)^{1/2}\rceil.$$
We do this with fixed absolute precision
$$t = s+1+\lfloor 2\log_2(2R+1)\rfloor,$$
i.e., we round to the nearest element of $2^{-t}\ZZ[i]$
at every step and ensure that the summands
are correct up to an additive error with
absolute value at most $2^{-t}$.
We use fast arithmetic as in~\cite{fastmultiplication}
to compute the individual terms.
\begin{theorem}\label{computingtheta}
  On input $j\in\{0,\ldots,15\}$,
  a positive integer~$s$,
  and a matrix $\widetilde{Z}\in\mathcal{B}$ with $|\widetilde{Z}-Z|< 2^{-t-1}$
  for some $Z\in\mathcal{H}_2$, the algorithm just described
  gives as output
  a complex number~$A$ with $|A-\theta_j(Z)|< 2^{-s}$
  in time $\widetilde{O}(s^2)$.
\end{theorem}
\begin{proof}
The number of terms to compute
is $O(R^2)=O(s)$, at precision $t=O(s)$ each.
With fast arithmetic, this takes time $\widetilde{O}(s)$ per term,
proving the running time.

A precision of $t+1$ in the input ensures that each term has an error
of at most $2^{-t}$.
The errors of the terms then add up to an error
with absolute value at most $(2R+1)^2 2^{-t}\leq 2^{-s-1}$.

The terms that are left out contribute 
$$L = \!\!\!\!\!\!\!\!\!\!\sum_{\substack{n\in\ZZ^2\\ \abs{n_1}>R\ \mathrm{or}\ \abs{n_2}>R}}
\!\!\!\!\!\!\!\!\!\!\exp(\pi i (n+c') Z (n+c')\transpose +2\pi i (n+c') {c''}\transpose )$$
to the error.
Let $m=n+c'$. We have $0\leq 2y_3\leq y_1\leq y_2$,
so $m Y m\transpose = m_1^2 y_1+2m_1m_2y_3+m_2^2y_2
\geq (m_1^2 - \abs{m_1m_2} + m_2^2) y_1
= \frac{1}{2}(\abs{m_1}-\abs{m_2})^2 y_1 + \frac{1}{2}(m_1^2+m_2^2)y_1\geq \frac{1}{2}(m_1^2+m_2^2)y_1$.
We conclude $$\abs{L} \leq  \!\!\!\!\!\!\!\!\!\!\sum_{\substack{n\in\ZZ^2\\ \abs{n_1}>R\ \mathrm{or}\ \abs{n_2}>R}} \!\!\!\!\!\!\!\!\!\!
\exp\left(-\frac{\pi}{2}(m_1^2+m_2^2)y_1\right)
\leq 8\left(\sum_{k=0}^{\infty} \exp(-\frac{\pi}{2} k^2 y_1)\right)
\left(\sum_{k=R}^{\infty} \exp(-\frac{\pi}{2} k^2 y_1)\right),$$
which is $\leq 2^{-s-1}$ for $y_1\geq\sqrt{3/4}$.
Both errors combined are $\leq 2^{-s}$.
\end{proof}

\begin{remark}\label{rem:dupont}
Note that this running time is quasi-quadratic,
 while Dupont's 
 (generalized AGM-)method \cite[Section 10.2]{dupont} is heuristically
 quasi-linear.
Proving correctness of Dupont's method, and
analysing the required precision and the running time,
is beyond the scope of this article.
\end{remark}
 
 \newcommand{\smabs}[1]{|#1|}
After computing approximations of the theta constants,
evaluating the absolute Igusa invariants
is straightforward.
First we evaluate each term in the formulas for $h_4$, $h_6$, $h_{10}$, $h_{12}$
of Lemma~\ref{lem:igusaintheta}
by multiplying theta constants one by one,
and then we evaluate the $h_k$ themselves by adding the terms one by one.
Finally, we invert $h_{10}$ and multiply the factors $h_k^{\pm 1}$ together. 
We do all this with absolute precision~$s$, i.e., with complex numbers
in $2^{-s}\ZZ[i]$, which we round back to $2^{-s}\ZZ[i]$ after every step.
The result is then as follows.
\begin{proposition}\label{prop:igusafromtheta}
Let $Z\in\mathcal{B}$ be a period matrix and 
$\widetilde{\theta}[c]\in 2^{-s}\ZZ[i]$
such that $|\theta[c](Z)-\widetilde{\theta}[c]|\leq 2^{-s}$.
Let $\widetilde{i_n}$ be obtained from the $\widetilde{\theta}[c]$ by the method
we have just described.

Let $u= 3+\pi(y_1+y_2-y_3) +\max\{2,-\log_2\abs{z_3}\}$. If $s$ is $>13+2u$, then
we get $|\widetilde{i_n}-i_n(Z)|< 2^{100+3u-s}$.
The running time is $\widetilde{O}(s)$ as $s$ tends to infinity,
where the implied constants do not depend on the input.
\end{proposition}
\begin{proof}
  For any term $A$ in $h_k$, let $A_i$ be $A$ after $i$
  factors have been multiplied together,
  so $\abs{A_i}\leq 1.56^{i}$.
  Let 
  $\widetilde{A}_i$ be the approximation
  of $A_i$ that is computed in the algorithm,
  and let $\widetilde{A}=\widetilde{A}_{2k}$
  be the approximation of $A$ obtained in this way.
  Then for the error $\epsilon(\widetilde{A}_i) = |\widetilde{A}_i-A_i|$,
  we have $\epsilon(\widetilde{A}_0)=0$ and
  $\epsilon(\widetilde{A}_{i+1})\leq
    1.56\epsilon(\widetilde{A}_{i})+1.56^{i}2^{-s}+2^{-s}$.
  By induction, we get $\epsilon(\widetilde{A}_{i})<2^{2+i-s}$,
  so that the approximation $\widetilde{A}$ of each term $A$
  in $h_{k}$ has an error of at most $\epsilon(\widetilde{A})<2^{2+2k-s}$.
  The error of $\widetilde{h}_{k}$ itself will therefore be less than
  $t_k 2^{2+2k-s}<2^{40-s}$, where $t_4=10$, $t_6=60$,
  $t_{10}=1$, and $t_{12}=15$.
  
  Next, we evaluate $h_{10}^{-1}$.
  Let $\widetilde{h}_{10}$
  be the approximation that we have just
  computed, so $\smabs{h_{10}-\widetilde{h}_{10}}<2^{12-s}$ and $\smabs{h_{10}}>2^{-u}$.
  As we have $s>13+u$, we find
\[
  \smabs{h_{10}^{-1}-\widetilde{h}_{10}^{-1}}=
  \frac{\smabs{h_{10}-\widetilde{h}_{10}}}
    {\smabs{h_{10}\widetilde{h}_{10}}}
  \leq  \frac{2^{12-s}}
          {2^{-u}2^{-u}(1-2^{12-s+u})}<2^{13+2u-s},
\]
  so we find an approximation of $h_{10}^{-1}$ with an error of
  at most $2^{13+2u-s}+2^{-s}< 2^{14+2u-s}$.
  
  Finally, we evaluate $i_1$, $i_2$, and $i_3$,
  and the bound on their errors
  follows from the absolute value and error bounds on $h_{k}$
  and $h_{10}^{-1}$.
 \end{proof}

\section{Bounding the period matrices}
\label{sec:boundonZ}
In this section, we prove the following result.
Here, the set $\mathcal{B}\subset\mathcal{H}_2$ is as
defined in Section~\ref{ssec:funddom}, and
contains the fundamental domain~$\mathcal{F}_2$.
\begin{theorem}\label{thm:boundonz}
   Let $Z\in\mathcal{B}$ be such that the principally polarized abelian
   variety corresponding to it has complex multiplication by $\mathcal{O}_K$.
   Then we have
   $m_2(\Im Z) \leq \frac{2}{3\sqrt{3}}\max\{2\Delta_0, \Delta_1^{1/2}\}$,
where
$\Delta_0$ is the discriminant of the real quadratic subfield
$K_0\subset K$, and $\Delta_1$ is the norm
of the relative discriminant of $K/K_0$.
\end{theorem}

Let $\mathfrak{a}$ and $\Phi=\{\phi_1,\phi_2\}$ be an ideal and CM-type of~$K$
corresponding to~$Z$ as in Section~\ref{ssec:quotientbyideal}.
Let $e$, $f$, $v$, $w\in K$ be a symplectic basis
of~$\mathfrak{a}$ giving rise to~$Z$ as in Section~\ref{ssec:periodmatrices}.
By scaling, we may assume $v=1$.
Write $w_k=\phi_k(w)$ for $k=1,2$.

\begin{lemma}\label{lem:firstbound}
We have
 $$|\det \Im Z| = |w_1-w_2|^{-2}\covol (\Phi(\mathfrak{a}))\quad\mbox{and}\quad
\covol (\Phi(\mathfrak{a}))=\frac{1}{4}N(\mathfrak{a})\Delta^{1/2}\leq \frac{1}{4}\Delta^{1/2}$$
\end{lemma}
\begin{proof}
Let $\varphi:\C^2\rightarrow \C^2$ be the $\C$-linear
map sending $(1,0)$ to $(1,1)=\Phi(1)$ and $(0,1)$ to $(w_1,w_2)=\Phi(w)$,
so $\varphi(Z\ZZ^2+\ZZ^2) = \Phi(\mathfrak{a})$.
As an $\RR$-linear map, it has determinant $|w_1-w_2|^{2}$.
We find $$|\det \Im Z| = \covol (Z\Z^2+\ZZ^2) = |w_1-w_2|^{-2}\covol (\Phi(\mathfrak{a})).$$
Moreover, we have $\covol (\Phi(\mathfrak{a}))= N(\mathfrak{a})\covol (\Phi(\OK))$,
where $\covol(\Phi(\OK))=\frac{1}{4}\Delta^{1/2}$.
Finally, our assumption $v=1$ implies that $\mathfrak{a}^{-1}$ is an integral ideal, so $N(\mathfrak{a})\leq 1$.
\end{proof}

\begin{lemma}\label{lem:boundonz1}
  Suppose $ w\not\in K_0$. Then we have
  $ |\det \Im Z| <\frac{1}{2} \Delta_0.$
\end{lemma}
\begin{proof}
  Write $w_k = x_k+iy_k$ and let $\xi$ be as in Section~\ref{ssec:quotientbyideal}.
  We have $\Tr_{K/\QQ}(\xi w) = E(\Phi(1),\Phi(w))=0$
  as $(e,f,1,w)$ is a symplectic basis.
  Write $\phi_k(\xi) = i\nu_k$, so $\nu_k$ is a positive real number.
  We get $0 
            = -2(\nu_1 y_1 +\nu_2 y_2)$, so $y_2 = -\frac{\nu_1}{\nu_2} y_1$.
  In particular, we have $|w_1-w_2|\geq |y_1-y_2| = |y_1|(1+\frac{\nu_1}{\nu_2})$.
  Analogously, we have $|w_1-w_2|\geq |y_2-y_1| = |y_2|(1+\frac{\nu_2}{\nu_1})$.
  Taking the product of these identities yields
  $|w_1-w_2|^2 \geq |y_1y_2|(2+\frac{\nu_1^2+\nu_2^2}{\nu_1\nu_2})>2|y_1y_2|$.

  On the other hand, $\mathfrak{a}$ contains $\OKO+ w\OKO$,
  which has covolume $\Delta_0 |y_1y_2|$.
  We get our result by inserting these values into the first equality
  of
  Lemma~\ref{lem:firstbound}.
\end{proof}

Write $Z=\left(\begin{array}{cc} z_1 & z_3 \\ z_3 & z_2\end{array}\right)$
and $z_k = x_k+iy_k$.
\begin{lemma}\label{lem:boundonz2}
  Suppose $w\in K_0$ and write $\mathfrak{b}=\ZZ+w\ZZ$.
  Then we have
$$|\det \Im Z| =\frac{1}{4}N_{K/\Q}(\mathfrak{a}^{-1}\B)^{-1}\Delta_1^{1/2}\leq \frac{1}{4}\Delta_1^{1/2},$$
where $N_{K/\Q}(\mathfrak{a}^{-1}\B)$ is an integer.
\end{lemma}
\begin{proof}
  Note that $\B= (K_0\cap \A)$ is a fractional $\OKO$-ideal
  with $\A\supset \OK\B$.
  We compute
\[N_{K/\Q}(\A)=N_{K_0/\Q}(\B)^2N_{K/\Q}(\A\B^{-1})
                  = |w_1-w_2|^{2}\Delta_0^{-1}N_{K/\Q}(\A^{-1}\B)^{-1}.\]
  We find the result by inserting this into the
  second equality of Lemma~\ref{lem:firstbound}.
\end{proof}  
\begin{proof}[{Proof of Theorem \ref{thm:boundonz}}]
Equations \eqref{eq:minendet} and (B) of Section~\ref{sec:up}
give $m_2(\Im Z) \leq \frac{4\sqrt{4}}{3\sqrt{3}}\det\Im Z$, hence Lemmas \ref{lem:boundonz1} and \ref{lem:boundonz2} prove the result.
\end{proof}

\begin{remark}
The bound of Theorem \ref{thm:boundonz} is not optimal.
For example, Corollary II.6.2 of the author's thesis
\cite{phdthesis} 
improves it
to $\max\{\frac{2\sqrt{2}}{\sqrt{3}\pi}\Delta_0, \frac{4}{9} \Delta_1^{1/4}\Delta_0^{1/2}\}$
using the \emph{Hilbert} upper half space
and multiple pages of computations.
However, we will be satisfied
with Theorem \ref{thm:boundonz}, as it is easier to prove
and not the bottleneck of our running time
analysis.
\end{remark}

\section{The degree of the class polynomials}
\label{sec:degree}

Let $K$ be a primitive quartic CM-field.
In this section we give asymptotic upper and lower bounds on
the degree of Igusa class polynomials of~$K$.
These bounds are not used in the algorithm itself, but are used
in the analysis of the algorithm.

Denote the class numbers of~$K$ and~$K_0$ by~$h$
and~$h_0$ respectively, and let $h_1=h/h_0$.
The degree of the Igusa class polynomials
$H_{K,n}$ for $n=1,2,3$ is the number $h'$ of isomorphism classes of curves
of genus~$2$ with CM by $\OK$.
By Lemma~\ref{lem:countas}
we have
$h'=h_1$ if $K$ is cyclic and $h'=2h_1$ otherwise.
The degree of the polynomials $\smash{\widehat{H}_{K,n}}$ is $h'-1$.
The following result gives an asymptotic bound on $h_1$,
and hence on the degree~$h'$.
\begin{lemma}[Louboutin]\label{lem:h1}
  There exist effective constants $d>0$ and $N$
  such that for all primitive
  quartic CM-fields~$K$ with $\Delta>N$, we have
  $$ \Delta_1^{1/2}\Delta_0^{1/2}(\log\Delta)^{-d}
   \leq h_1\leq\Delta_1^{1/2}\Delta_0^{1/2}(\log\Delta)^d.$$
\end{lemma}
\begin{proof}
  Louboutin
\cite[Theorem 14]{louboutinlower}
  gives
  bounds
  $$\abs{\frac{\log h_1}{\log(\Delta_1\Delta_0)}-\frac{1}{2}}
   \leq d\frac{\log\log\Delta}{\log \Delta}$$
  for $\Delta>N$. Multiply through by
  $\log(\Delta_1\Delta_0)$ and note
  $d \frac{\log\log\Delta}{\log\Delta} \log(\Delta_1\Delta_0) < d\log\log\Delta$.
\end{proof}

\section{Denominators}
\label{sec:denominators}

Let $K$ be a primitive quartic CM-field.
In this section we give upper bounds on
the denominators of the Igusa class polynomials of~$K$.
By the \emph{denominator}
of a polynomial $f\in\Q[X]$, we mean the
smallest positive integer $c$ such that $cf$ is in $\Z[X]$.

\subsection{Background}

A prime $p$ occurs in the denominator of $H_{K,n}$
only if there is a curve $C$ with CM by $\OK$ 
such
that $C$ has \emph{bad reduction} at a
prime $\mathfrak{p}$ over~$p$.
It is known
that abelian varieties with complex multiplication have
potential
\emph{good} reduction at all primes,
but this does not imply that Jacobians
reduce as Jacobians: the reduction of the
Jacobian of a smooth curve~$C$ of genus two can be a
polarized product of elliptic curves $E_1\times E_2$.
The reduction of $C$ is then the
union of those elliptic curves intersecting
transversely.
For details, we refer
to Goren and
Lauter~\cite{goren-lauter,gorenlautertoappear},
who study this phenomenon and use the
embedding $$\OK\rightarrow \mathrm{End}(E_1\times E_2)$$
to bound both $p$
and the valuation of the denominator of $H_{K,n}$ at~$p$.

We use the bounds of Goren and
Lauter
which hold in general, but are
expected to be far from asymptotically optimal,
in our running time analysis.
The bounds of Bruinier and Yang~\cite{bruinier-yang, yang}
are better, but are proven only for very special quartic CM-fields.

\subsection{Statement of the results}

Goren and Lauter~\cite{goren-lauter,gorenlautertoappear}
give their bounds
in terms of integers $a$, $b$, $d$ such that
$K$ is given by $K=\Q(\smash{\sqrt{-a+b\sqrt{d}}})$.
For~$d$, one can take the discriminant
$d=\Delta_0$ of the real quadratic subfield $K_0$.
We will prove in Lemma~\ref{denomintrinsic} below
that one can take $a<8\pi^{-1}(\Delta_1\Delta_0)^{1/2}$,
where $\Delta_1=N_{K_0/\Q}(\Delta_{K/K_0})$
is the norm of the relative discriminant.
The denominator itself does not depend
on the choice
of~$a$, so we can replace $a$ by this bound on $a$
in all denominator bounds below.

The main result of this section is the following.
\begin{theorem}\label{thm:denominators}
  Let $K$ be a primitive quartic CM-field
  and write
$$K=\Q\Big(\textstyle{\sqrt{-a+b\sqrt{d}}}\Big)\quad\mbox{with}\quad a,b,d\in\Z.$$
  The denominator of each of the Igusa class polynomials
  of $K$
  divides $D=2^{24h'}D_1^2$ for 
$$D_1=\Bigg( \prod_{{\substack{p<4da^2 \\ p\ \textit{prime}}}}
  p^{\lfloor 4f(p)(1+ \log (2da^2)/\log p)\rfloor} \Bigg)^{h'},$$
  where $f(p)$ is given by 
  $f(p)=8$ if $p$ ramifies in $K/\Q$ and satisfies $p\leq 3$,
  and given by $f(p)=1$ otherwise.

  Furthermore, the result above remains true if we replace~$d$
  by $\Delta_0$ and~$a$ by $\lfloor 8\pi^{-1}(\Delta_1\Delta_0)^{1/2}\rfloor$
  in the definition of~$D_1$.
  We then have
  $\log D=\widetilde{O}(h'\Delta)=\widetilde{O}(\Delta_1^{3/2}\Delta_0^{5/2})$ as $\Delta$ tends to
  infinity.
\end{theorem}
We will prove this result below.
\begin{remark}\label{rem:c3andk}
Theorem~\ref{thm:denominators} as stated holds for the absolute Igusa invariants
$i_1$, $i_2$, $i_3$ of Section~\ref{sec:igusa}.
For another choice of a set $S$ of absolute Igusa invariants, 
take positive integers $c_3$ and $k$
such that  $c_3(2^{-12}I_{10})^k S$
consists of modular forms of degree $k$ with integral Fourier expansion.
Then the denominator divides $c_3^{h'}D_1^k$.
See the proof below of Theorem~\ref{thm:denominators} for details.
\end{remark}

%

\begin{remark}\label{rem:decompcondition}
  It follows from Goren~\cite[Thms.~1 and~2]{goren}
  that Theorem~\ref{thm:denominators} remains
  true if one restricts
  in the product defining $D_1$
  to primes $p$ that divide $2\cdot3\cdot c_3\Delta$ or
  factor as a product of two
  prime ideals in $\mathcal{O}_K$.
\end{remark}

\subsection{The bounds as stated by Goren and Lauter}

The first part of the proof of Theorem~\ref{thm:denominators}
is the following bound on the primes that occur in the denominator.
\begin{lemma}[{Goren and Lauter \cite{goren-lauter}}]\label{lem:gorenlauter}
  The coefficients of each of the polynomials $H_{K,n}(X)$ and $\widehat{H}_{K,n}$
  for $K=\Q(\smash{\sqrt{-a+b\sqrt{d}}})$ a primitive quartic CM-field
  are $S$-integers,
  where $S$ is the set of primes
  smaller than $4 d a^2$.
\end{lemma}
\begin{proof}
 Corollary 5.2.1 of \cite{goren-lauter} is this result with
 $4 d^2 a^2$ instead of $4 d a^2$.
 We can however adapt the proof as follows to remove a factor~$d$.
 In \cite[Corollary 2.1.2]{goren-lauter}, it suffices
 to have only $N(k_1)N(k_2)<p/4$ in order for two elements $k_1$ and $k_2$
 of the quaternion order ramified in $p$ and infinity to commute.
 Then, in the proof of \cite[Theorem 3.0.4]{goren-lauter},
 it suffices to take as hypothesis only
 $p>d(\mathrm{Tr}(r))^2$.
 As we have $d(\mathrm{Tr}(r))^2\geq d\delta_1\delta_2\geq N(x)N(by^\vee)$, this implies
 that $x$ and $by^\vee$ are in the same imaginary quadratic field $K_1$.
 As in the original proof, this implies that $ywy^\vee$ is also contained in $K_1$
 and hence $\psi(\sqrt{r})\in M_2(K_1)$,
 so there is a morphism $K=\Q(\sqrt{r})\mapsto M_2(K_1)$,
 contradicting primitivity of~$K$.
\end{proof}
\begin{remark}
 Lemma \ref{lem:gorenlauter} as phrased above is for class polynomials
  defined in terms of the invariants $i_1$, $i_2$, $i_3$ of Section~\ref{sec:igusa}.
  If other invariants are used, then the result is still valid
  if we include the primes dividing $c_3$
  of Remark~\ref{rem:c3andk} in~$S$.
\end{remark}

Recent results
bound the exponents to which primes occur
in the denominator as follows.
\begin{lemma}[{Goren-Lauter \cite{gorenlautertoappear}}]\label{lem:expdenom}
Let $K$ be a primitive quartic CM-field and
$C/\C$ a curve of genus~$2$ that has
CM by~$\OK$.
Let $v$ be a non-archimedean valuation 
of $L(i_n(C))$, normalized with respect to $\Q$ in the sense that
$v(\Q^*)=\Z$ holds, and let $e$ be its ramification index
(so $ev$ is normalized with respect to $L(i_n(C))$.
Let $k$ and $c_3$ be as in Remark~\ref{rem:c3andk}.

Then we have
\begin{align}
   -v(i_n(C)) &\leq  4 k (\log(2da^2)/\log(p) +1)+v(c_3) &
         &\mbox{if $e\leq p-1$, and}\nonumber\\
   -v(i_n(C)) &\leq  4 k (8\log(2da^2)/\log(p)+2)+v(c_3) &
         &\mbox{otherwise.}\label{eq:valuationbound}
\end{align}
Moreover, $e\leq p-1$ is automatic for $p\not= 2$,~$3$.
\end{lemma}
\begin{proof}
Theorem 7.0.4 of Goren and Lauter~\cite{gorenlautertoappear}
gives the valuation bounds.
(Here we use that the gcd of the Fourier coefficients
of $h_{10}$ is $2^{12}$, a fact that can be
found in e.g.~\cite[Proof of Corollary~5.2.1]{goren-lauter}
and~\cite[Appendix 1]{phdthesis}, hence $\Delta$ in the
notation of~\cite{gorenlautertoappear} is $2^{-12}h_{10}$.)

Next, we show $e\leq 4$ for $p>2$.
Let $L\subset \C$ be isomorphic to the normal closure of~$K$,
let $\Phi$ be the CM-type of $C$
and $K^{\mathrm{r}}\subset L$ its reflex field.
The extension $K^{\mathrm{r}}(i_n(C))/K^{\mathrm{r}}$
is unramified by
the main theorem of complex multiplication~\cite[Main Theorem~1 in \S15.3 in Chap.~IV]{shimura-taniyama}.
In particular, the ramification index of any prime
in $L(i_n(C))/\Q$ is at most its ramification
index in $L/\Q$.
By Lemma~\ref{cmintro:lem:geq2cmtypes},
the field $L$ has degree $4$ over $\Q$ or has degree
$2$ over a biquadratic subfield, hence we have
$e\leq 4$ for $p>~2$.
\end{proof}

\subsection{The bounds in terms of discriminants}

Lemmas \ref{lem:gorenlauter} and \ref{lem:expdenom}
hold for any representation of $K$ of
the form $K=\Q(\smash{\sqrt{-a+b\sqrt{d}}})$,
hence in particular for such a representation
with $da^2$ minimal.
The following result gives a lower and an upper bound on the
minimal~$da^2$.
\begin{lemma}\label{denomintrinsic}
  Let $K$ be a quartic CM-field with discriminant $\Delta$
  and let $\Delta_0$ be the discriminant of the real quadratic
  subfield~$K_0$.

  For all $a,b, d\in\Z$ such that $K=\Q(\smash{\sqrt{-a+b\sqrt{d}}})$ holds,
  we have $a^2 > \Delta_1$ and $d \geq \frac{1}{4}\Delta_0$. Conversely,
  there exist such $a,b,d \in \Z$
  with $d=\Delta_0$ and 
  $a^2<(\frac{8}{\pi})^2  \Delta_1\Delta_0$.
\end{lemma}
\begin{proof}
The lower bounds are trivial, because
$\Delta_0$ divides $4d$ and
$\Delta_1$ divides $a^2-b^2d\leq a^2$. For
the upper bound, we show the existence of a
suitable element $-a+b\sqrt{\Delta_0}$ using a
geometry of numbers argument.

We identify $K\otimes_{\Q} \RR$ with $\C^2$ via its pair
of infinite primes.
Then $\mathcal{O}_K$ is a lattice in $\C^2$
of covolume $2^{-2}\sqrt{\Delta}$.
Let
$\omega_1,\omega_2$ be a $\Z$-basis of $\mathcal{O}_{K_0}$,
and consider the open parallelogram
$\omega_1(-1,1)+\omega_2(-1,1)\subset \mathcal{O}_{K_0}\otimes\RR\cong \RR^2$.
We define the
open convex symmetric region
$$V_Y=\{x\in\C^2:\Re(x)\in\omega_1(-1,1)+\omega_2(-1,1),
(\Im x_1)^2+(\Im x_2)^2<Y\}.$$ Then
$\mathrm{vol}(V_Y)=4\pi\sqrt{\Delta_0}Y$
and by Minkowski's convex body theorem,
$V_Y$ contains a non-zero element $\alpha\in\mathcal{O}_K$ if we have
$$\mathrm{vol}(V_Y)>2^4\mathop{\mathrm{covol}}\OK=4\sqrt{\Delta}.$$
We pick
$Y=\sqrt{\Delta_1\Delta_0}\pi^{-1}+\epsilon$, so that $\alpha$ exists.

Let
$r=4(\alpha-\overline{\alpha})^2$, which is of the form
$-a+b\sqrt{\Delta_0}$ with integers $a$ and~$b$.
Now $a=\frac{1}{2}\abs{r_1+r_2}=2(2\Im x_1)^2+2(2\Im x_2)^2<8Y=8\sqrt{\Delta_1\Delta_0}\pi^{-1}+8\epsilon$.
As $a$ is in the discrete set $\Z$, and we can take $\epsilon$ arbitrarily
close to~$0$, we find that
we can even get $a\leq 8\sqrt{\Delta_1\Delta_0}\pi^{-1}$
and hence $a^2\leq (\frac{8}{\pi})^2\Delta_1\Delta_0$.
\end{proof}

\begin{proof}[Proof of Theorem \ref{thm:denominators}]
Lemma \ref{lem:gorenlauter}
proves that the denominator of the Igusa class polynomials is divisible
only by primes dividing~$D$.

Before we bound the valuations on these primes,
we determine $c_3$ and $k$ from Remark~\ref{rem:c3andk}
for our choice of Igusa invariants.
The theta constants have integral Fourier coefficients by definition,
hence so do $\homigu{2}{}\homigu{10}{}$, $\homigu{4}{}$, $\homigu{6}{}$,
and $\homigu{10}{}$ by Lemma~\ref{lem:igusaintheta}.
This shows that with our choice of absolute Igusa invariants,
we can use $c_3=2^{24}$ and~$k=2$.
(A longer analysis 
improves the $24$ to $14$ for our invariants,
but we will not use that; see
\cite[Appendix~1]{phdthesis}.)

Next, let $v$ be any normalized non-archimedean valuation of $H_{K^{\mathrm{r}}}$
and $c$ any coefficient of $H_{K,n}$ or $\widehat{H}_{K,n}$.
Then $c$ is a sum of products, where each product consists of at most $h'$
factors $i_n(C)$ for certain $n$'s and $C$'s.
This shows that $-v(c)$ is at most $h'$ times the right hand side
of~\eqref{eq:valuationbound},
hence $v(Dc)\geq 0$.
As this holds for all~$v$,
it follows that $Dc$ is an integer.
This concludes the proof that $DH_{K,n}$ and $D\widehat{H}_{K,n}$ are in $\Z[X]$.


The fact that we can replace $a$ and $d$ as in the theorem is Lemma~\ref{denomintrinsic}.
Next, we prove the asymptotic bound on~$D$.
Note that the exponent of every prime in $D^{1/h'}$ is linear in $\log \Delta$,
as is the bit size of every prime divisor of~$D$.
Therefore, $\log D$ is $\widetilde{O}(h' N)$, where $N=O(\Delta)$ is the number of prime divisors of~$D$,
which finishes the proof of Theorem~\ref{thm:denominators}.
\end{proof}

\section{On our choice of absolute Igusa invariants}
\label{sec:invariants}

Next, we motivate our choice of absolute Igusa invariants
$i_1$, $i_2$, $i_3$ over Spallek's $j_1$, $j_2$, $j_3$~\cite{spallek}
and over ECHIDNA's~\cite{echidna}
(denoted $i_1$, $i_2$, $i_3$
in~\cite{echidna}, but which we will denote
by $i_4$, $i_6$, $i_7$).
We chose our invariants
to be of the form $i=h_4^ah_6^bh_{12}^c/h_{10}^k$
with $a,b,c,k$ non-negative integers satisfying $4a+6b+12c=10k$.
Such invariants form a $\QQ$-basis of the $\QQ$-algebra
of all invariants of genus-two curves.

Among those, we took $k>0$ as small as possible.
The motivation for taking $k$ to be small is that
both our upper bound on the absolute value of $i(C)$
and Goren and Lauter's bound on the denominator of $i(C)$ grow with~$k$
(see respectively Corollary~\ref{cor:thm2numanddenom}
and Remark~\ref{rem:c3andk}).
For $k=1$, this yields only~$i_1$,
while for $k=2$, it gives only $i_1^2$, $i_2$, and~$i_3$.
So our invariants are chosen such that these bounds
are as good as possible.

Next, we show that our invariants are also good experimentally.
All invariants mentioned above 
are listed in the following table,
and we explain the final column below.
\[\begin{array}{|lclclcl|l|}
\hline
i_1& & &=&\homigu{4}{}\homigu{6}{\prime}/\homigu{10}{}&=&h_{4}h_6/h_{10}& 
      0.6686 \\ 
i_{4}& & &=&\homigu{4}{}\homigu{6}{}/\homigu{10}{} 
   &=&  h_4^{\vphantom{2}}h_{16}^{\vphantom{2}}/h_{10}^2     &
      1    \\ 
i_2& & &=& \homigu{2}{}\homigu{4}{2}/\homigu{10}{} 
   &=&  h_4^2h_{12}^{\vphantom{2}}/h_{10}^2     & 
     1.0294 \\ 
i_3& & &=& \homigu{4}{5}/\homigu{10}{2}
   &=& h_4^5/h_{10}^2 
    &1.4203 \\ 
i_7 &=& 2^{-3} j_3&=& \homigu{2}{2}\homigu{6}{}/\homigu{10}{}
   &=&  h_{12}^2h_{16}^{\vphantom{2}}/h_{10}^4& 
           1.7799 \\ 
i_6 &=& 2^{-1} j_2 &=& \homigu{2}{3}\homigu{4}{}/\homigu{10}{}
   &=&  h_4^{\vphantom{2}}h_{12}^3/h_{10}^4&
         1.7949 \\ 
i_5 &=& 2^3 j_1 &=& \homigu{2}{5}/\homigu{10}{}
   &=& h_{12}^5/h_{10}^6 &
    2.5921 \\ 
\hline
\end{array}
\]
For each of these invariants~$i$,
and each of more than a thousand quartic CM-fields~$K$
in the ECHIDNA database~\cite{echidna}, we computed
the class polynomial $H_{K,i}$ with $i(C)$ as roots,
where~$C$ ranges over curves of genus two with CM by the maximal order of~$K$.
We then scaled $H_{K,i}$ to make it minimal with integer coefficients,
and took the largest absolute value of those coefficients as a measure of the size of~$H_{K,i}$
(call it $s(K,i)$).
We plotted $\log(s(K,i))$ relative to $\log(s(K,i_4))$ for each~$i$,
and computed a least-squares fitting linear function using Sage~\cite{sage}.
The rightmost column of the table is the slope of this linear function.
More details, including the plots themselves and some
additional invariants can be found in
the author's thesis~\cite[Appendix~3]{phdthesis}.
The powers of~$2$ that we multiplied Spallek's invariants with did
not influence these numbers much.

If we use this final column as a measure for the size of the class polynomials,
then that makes~$i_1$ a clear winner.
The functions~$i_2$ and~$i_4$ have a joint second place, but
we can use only one, as they satisfy $i_2 = \frac{1}{2} (i_1 - 3 i_4)$.
Our choice for~$i_2$ over~$i_4$ was arbitrary,
and then~$i_3$ is obviously the next invariant to take.

This shows experimentally that our choice of invariants
performs better in practice than the other choices.
We could still also scale $i$ with a constant. However, this
has a relatively small effect, and our invariants without constants are easier to remember.
Also, we found that the natural scaling (making sure the gcd of the
Fourier coefficients is 1 for both the numerator and the denominator)
yields worse sizes in practice than using no scaling at all.

\section{Recovering a polynomial from its roots}
\label{sec:polyfromroots}

At this point, we know how to find approximations
of the roots of the polynomial $H_{K,n}(X)$, and
we wish to combine these into approximations of the
coefficients of $H_{K,n}(X)$.
In other words, we need to take the product of a
set of linear polynomials.

\subsection{Numerically multiplying many polynomials}
\label{ssec:polyfromroots}

We compute the product of a set of linear
polynomials by
arranging them in a binary tree, and computing the products
of pairs of polynomials using fast multiplication.
This method is well known, and a complete
analysis of its running time and rounding errors
is given by Kirrinnis~\cite{kirrinnis}.

Define the norm of a polynomial $p=\sum a_k x^k\in \CC[x]$
to be $\abs{p}=\abs{p}_1=\sum \abs{a_k}$.
Let $p_1,\ldots,p_n$ be linear polynomials such that
$\abs{p_j}\leq 2^{t_j}$ holds with $t_j\geq 1$, and let $t=\sum t_j$.
In particular, if $p_j = (X-z_j)$, take $t_j \geq \max\{\log_2 (\abs{z_j}+1),1\}$.
\begin{theorem}[{%
Kirrinnis \cite{kirrinnis}}]\label{thm:polynomialmultiplication}
 There exists an explicit algorithm,
independent of the data mentioned above,
with the following input, output and running time.\\
\textup{\textbf{Input:}} Positive integers $n$, $s$ and $t_1,\ldots,t_n$, and
linear polynomials $\widetilde{p}_1,\ldots,\widetilde{p}_n$ satisfying
$$
\abs{\widetilde{p}_j-p_j}< 2^{-(s+ t - t_j + 2\lceil\log_2 n\rceil)}.$$
\textup{\textbf{Output:}} A polynomial $\widetilde{p}$ satisfying
$\abs{p_1\cdots p_n - \widetilde{p}}<2^{-s}$,\\
\textup{\textbf{Running time:}} $O(\psi(n\cdot \log n\cdot (s+t)))$,
where $\psi(k) = O(k\log k \log\log k)$ is the time needed
for multiplication of two $k$-bit integers.
\end{theorem}
\begin{proof}
  We reduce to the case $t_j=1$
  by the substitution
  $t_j\mapsto 1$,
  $t\mapsto n$,
  $s\mapsto s+t-n$,
  $p_j\mapsto 2^{-t_j+1}p_j$,
  $\widetilde{p}_j\mapsto 2^{-t_j+1}\widetilde{p}_j$,
  $\widetilde{p}\mapsto 2^{-t+n}\widetilde{p}$.
  Note that it takes linear time to move the point
  $t_j-1$ places to the left in $\widetilde{p}_j$ and to
  move it back to its correct position in the output~$\widetilde{p}$.

  For the case $t_j=1$, this result is a special case of Algorithm~5.1
  of~\cite{kirrinnis}.
  To see this in the notation of loc.~cit., note $t=n$, let $l=n$,
  and let $\mathbf{n}=(n_1,\ldots,n_l)=(1,\ldots,1)$.
  The definitions of $H_1(\mathbf{n})$
  and $d_j(\mathbf{n})$ can be found
  on page~407 of~\cite{kirrinnis},
  and it follows that in our case
  $d_j(\mathbf{n})\leq \lceil \log_2 n\rceil$ and
  $H_1(\mathbf{n})\leq n\lceil \log_2 n\rceil$ hold.
  For $\psi$, see \cite[p.~383]{kirrinnis},
  and for $\abs{p}$ and $\Pi_n$, see \cite[p.~381]{kirrinnis}.
\end{proof}
\begin{remark}
  The restriction to linear input polynomials is
  only to make the bounds on the running time
  and the required input precision easier to state.
  It is not present in~\cite{kirrinnis}.
\end{remark}
\begin{remark}
  For more details about the history of the algorithm,
  see~\cite[Section~3.2]{kirrinnis}.
\end{remark}

\subsection{Recognizing rational coefficients}
\label{ssec:rationalfromapprox}

There are various ways of recognizing a polynomial $f\in\QQ[X]$
from an approximation~$\widetilde{f}$.
If one knows an integer $D$ such that the denominator of $f$ divides~$D$, and the error
$|\widetilde{f}-f|$ is less than $(2D)^{-1}$, then
$Df$ is obtained from $D\widetilde{f}$ by rounding the coefficients
to the nearest integers.

Other methods to compute $f$ from $\widetilde{f}$ are based on continued fractions,
where the coefficients of $f$ are obtained via the continued fraction expansion of the coefficients
of~$\widetilde{f}$, or on the LLL-algorithm, where the coefficients of an
integral multiple of $f$ arise as coordinates of a small vector in a lattice~\cite[Section 7]{lattices}.
These methods have the advantage that only a bound~$B$ on the denominator needs to be known,
instead of an actual multiple~$D$.
This is very useful in practical implementations,
because one can guess a small value for~$B$, which may be much smaller than any
easily computable proven~$D$.
In the case of Igusa class polynomials,
there exist a few good heuristic checks of the output
when using a non-proven bound~$B$,
such as smoothness of the denominators, and successfulness of
CM constructions of abelian surfaces over finite fields.

For our purpose of giving a proven running time bound however,
we prefer the first method of rounding~$D\widetilde{f}$,
since it is easy to analyze and asymptotically fast.

It 
takes time $\widetilde{O}(\log D)$
to compute $D$ of Theorem~\ref{thm:denominators}
using sieving to find the primes and a binary tree to multiply them together.
We conclude that we can compute $H_{K,n}$ from $\widetilde{H}_{K,n}$
in time $\widetilde{O}(\log D)$ plus time linear in the bit size of $\widetilde{H}_{K,n}$,
provided that we have $|\widetilde{H}_{K,n}-H_{K,n}|<(2D)^{-1}$.

\section{The algorithm}

We now have all the required ingredients for
the algorithm and a proof of our main theorem.

\label{sec:algorithm}
\begin{algorithm}\label{alg:mainalgorithm}\hfil\break
  \textbf{Input:} A positive quadratic fundamental
  discriminant $\Delta_0$ and positive integers $a$ and $b$
   such that the field $K=\krootsmall$
   is a primitive quartic CM-field
   of discriminant greater than~$a$.
  \textbf{Output:} The Igusa class polynomials $H_{K,n}$ for $n=1,2,3$.
  \begin{enumerate}
   \item \label{itm:zbasis} Compute a $\ZZ$-basis of $\OK$ using the algorithm of
     Buchmann and Lenstra \cite{buchmann-lenstra}
    and use this to compute the discriminant $\Delta$ of~$K$.
   \item \label{itm:representatives} Compute a complete set $\{A_1,\ldots,A_{h'}\}$
     of representatives of the $h'$ isomorphism
     classes of principally polarized abelian surfaces over $\CC$ with CM by~$\OK$,
     using Algorithm~\ref{alg:vwamelenalgorithm}.
     Here each $A_j$ is given by a triple $(\Phi_j,\A_j,\xi_j)$ as in
     Section~\ref{sec:enumeratingdetails}.
   \item \label{itm:gettingd} From $\Delta$ and~$h'$,
     compute $D$ such that
    $D H_{K,n}$ is in $\ZZ[X]$ for $n=1,2,3$,
    as in Section~\ref{ssec:rationalfromapprox}.
   \item \label{itm:basisandbounds} For $j=1,\ldots, h'$, do the following.
\begin{enumerate}
\item \label{itm:sympbasis} Compute a symplectic basis of $\A_j$
      using Algorithm~\ref{alg:gramschmidt}.
      This provides us with a 
      period matrix $W_j\in\HH_2\cap \mathrm{Mat}_2(L)$,
      where $L\subset \CC$ is the normal closure of~$K$.
   \item \label{itm:funddom} Replace the period matrix $W_j$
      by an $\mathrm{Sp}_4(\ZZ)$-equivalent period matrix
      $Z_j\in \calF_2\cap \mathrm{Mat}_2(L)$,
      using Algorithm~\ref{bigreductionalgorithm}.
   \item \label{itm:bounds} 
      Let $u_j=\lceil 3+(y_1+y_2-y_3)\pi +\max\{2,-\log_2\abs{z_3}\}\rceil$,
      where $$ Z_j=\tbt{z_1 & z_3 \\ z_3 & z_2}\quad\mbox{and}\quad
      y_k=\Im z_k\quad (k=1,2,3).$$
\end{enumerate}
\item \label{itm:pn} Let $P_{\mathrm{basic}} = \lceil\log_2 D\rceil + 2\sum_{i} u_i + 2\lceil \log_2 n\rceil + 59 h' - 58$.
   \item \label{itm:approxin} For $j=1,\ldots,h'$, do the following.
\begin{enumerate}
\item \label{itm:homiguprecision} \label{itm:evaluatetheta}
      Let $P_{\mathrm{theta}}(j) = P_{\mathrm{basic}}+100+u_j$
      and evaluate the theta constants in~$Z_j$ 
      with error at most $2^{-P_{\mathrm{theta}}(j)}$
      as in Theorem~\ref{computingtheta}.
\item \label{itm:igusainvariants}
  Use Proposition~\ref{prop:igusafromtheta} to evaluate
   $i_n(A_j)$ for ($n=1,2,3$) with error less than
   $2^{-P_{\mathrm{Igusa}}(j)}$,
   where $P_{\mathrm{Igusa}}(j) = P_{\mathrm{basic}}-2u_j$.
\end{enumerate}
\item For $n=1,2,3$, do the following.
\begin{enumerate}
   \item \label{itm:dopolyfromroots} Use
          the algorithm of Theorem~\ref{thm:polynomialmultiplication}
          to compute an approximation $\widetilde{H}_{K,n}$
     of $H_{K,n}$ for $n=1,2,3$ from
     the approximations of Igusa invariants of step~\ref{itm:igusainvariants}.
   \item\label{itm:rounding} Compute $DH_{K,n}$ by rounding the coefficients of $D\widetilde{H}_{K,n}$ to nearest integers.
   \item Output $H_{K,n}$.
\end{enumerate}
  \end{enumerate}
This finishes the algorithm for the polynomials~$H_{K,n}$.
The interpolating polynomials $\widehat{H}_{K,n}$ ($n=2,3$) of Section~\ref{ssec:hhat}
can be computed from the approximations of $i_n(C)$ and $i_1(C)$
using Algorithm~10.9 of~\cite{gathen-gerhard}
(see also \cite[Section~4]{enge-morain}).
However, instead of doing a detailed rounding error analysis of that algorithm,
we give a more naive and slower algorithm that is still
dominated by the running time
in our Main Theorem.
To compute the polynomials $\widehat{H}_{K,n}$, we simply
modify step~\ref{itm:dopolyfromroots} as follows:
\begin{enumerate}
\item 
  Approximate each summand in the definition of the polynomial $\widehat{H}_{K,n}$
  using the algorithm of Theorem~\ref{thm:polynomialmultiplication}.
\item 
  Approximate $\widehat{H}_{K,n}$ by adding its summands.
\end{enumerate}
\end{algorithm}
We now recall and prove the main theorem.
\begin{mainthm}
  Algorithm \ref{alg:mainalgorithm} computes $H_{K,n}$ ($n=1,2,3$)
  for any primitive quartic {CM-field}~$K$.
  It has a running time of
  $\widetilde{O}(\Delta_1^{7/2}\Delta_0^{11/2})$
  and the bit size of the output
  is $\widetilde{O}(\Delta_{1}^2\Delta_0^{3})$.
\end{mainthm}
\begin{proof}
We start by proving that the output is correct.
By Theorem~\ref{computingtheta}, we obtain the theta constants
evaluated at $Z_j$ with an error of at most $2^{-P_{\mathrm{theta}}(j)}$.
Then Proposition~\ref{prop:igusafromtheta} shows
that the absolute Igusa invariants $i_n(Z_j)$ are correct with
an error less than $2^{-P_{\mathrm{Igusa}}(j)}$,
as $P_{\mathrm{Igusa}}(j) = P_{\mathrm{theta}}(j) - 100 - 3u_j$.

Next, we obtain $\widetilde{H}_{K,n}$ by multiplying
together the~$h'$ linear polynomials $p_j=(X-z_j)$, where
$z_j = i_n(Z_j)$, using the algorithm of Theorem~\ref{thm:polynomialmultiplication}.
In the notation of that theorem,
take $s=1+\lceil \log_2 D \rceil$, $n=h'$
and $t_j = 2u_j + 59$. We check that the hypotheses on the input
are satisfied.
Indeed, the error of~$p_j$ is less than $2^{-P_{\mathrm{Igusa}}(j)}$
and we have $-P_{\mathrm{Igusa}}(j) = -(s-t_j + \sum_{i} t_i+2\lceil \log_2 n\rceil)$.
Next, the norm of~$p_j$ is $|p_j| = |i_n(Z_j)|+1$
and we have $\log_2|i_n(Z_j)| \leq 2u_j + 58$, so~$|p_j|\leq 2^{t_j}$.

As the hypotheses of Theorem~\ref{thm:polynomialmultiplication}
are verified, we conclude that the output $\widetilde{H}_{K,n}$
has an error of at most~$2^{-s}<(2D)^{-1}$, so
that we indeed obtain $DH_{K,n}$ when rounding
the coefficients of $D\widetilde{H}_{K,n}$
to nearest integers.
%
This proves that the output of Algorithm \ref{alg:mainalgorithm} is correct.

Next, we bound the precisions $P_{\mathrm{Igusa}}(j)$ and $P_{\mathrm{theta}}(j)$
so that we can bound the running time.
We start by bounding~$u_j$, for which
we need an upper bound on $y_1+y_2-y_3$
and a lower bound on $z_3$.
We have $y_2\geq y_1$ and $y_3\geq 0$, 
hence $y_1+y_2-y_3\leq 2y_2$, and
Theorem~\ref{thm:boundonz} gives the upper bound
$y_2 \leq \frac{2}{3\sqrt{3}}\max\{2\Delta_0, \Delta_1^{1/2}\}$.

We claim that the off-diagonal entry $z_3$ of $Z_j\in\HH_2$ is non-zero.
Indeed, if $z_3=0$, then $Z_j=\diag(z_1,z_2)$ with $z_1$, $z_2\in\HH=\HH_1$
and $A_j$ is the product of the  elliptic
curves corresponding to $z_1$ and~$z_2$,
contradicting the fact that $A_j$ is simple
(Theorem~\ref{cmintro:thm:triplessimple}.\ref{cmintro:itm:primitive}).
Corollary~\ref{cor:movingup} now gives an upper bound
on $\log (1/z_3)$, which is polynomial in $\log \Delta$ by
the last sentence of
Section~\ref{ssec:computesympbasis}.

We now have 
$$u_j = O(\max\{\Delta_0, \smash{\Delta_1^{\smash{1/2}}}\}), \quad h'=\otilde{1/2}{1/2},$$
and by Theorem~\ref{thm:denominators} also $\log D=\otilde{3/2}{5/2}$.
We find that $P_{\mathrm{Igusa}}(j)$ is dominated by our bounds on $\log D$,
hence we have
$P_{\mathrm{Igusa}}(j)=\otilde{3/2}{5/2}$ and also $P_{\mathrm{theta}}(j)=\otilde{3/2}{5/2}$.

Now that we have bounds on the precision,
we can bound the running time.
Under the assumption that $K$ is given
as 
$K=\krootsmall$, where $\Delta_0$ is a positive fundamental discriminant
and $a,b$ are positive integers such that $a<\Delta_0$,
we can factor $(a^2-b^2\Delta_0^{\vphantom{2}})\Delta_0^2$ and hence find the ring of integers
in step~\ref{itm:zbasis} in time~$O(\Delta)$.

As shown in Section~\ref{sec:classgroups}, step~\ref{itm:representatives} takes time $\smash{\widetilde{O}(\Delta_{\vphantom{1}}^{1/2})}$.
Step~\ref{itm:gettingd} takes time
$\widetilde{O}(\log D)=\otilde{3/2}{5/2}$.

For every~$j$, step~\ref{itm:sympbasis} takes time polynomial in $\log\Delta$
by Theorem~\ref{thm:movingup} and the last sentence of
Section~\ref{ssec:computesympbasis}.
The same holds for steps~\ref{itm:funddom} and~\ref{itm:bounds} and
each summand of step~\ref{itm:pn}.
The number of iterations or summands of these steps is
$h'=\otilde{1/2}{1/2}$
by Lemmas \ref{lem:h1} and~\ref{lem:countas}.
In particular, steps \ref{itm:basisandbounds} and~\ref{itm:pn}
take time $\otilde{1/2}{1/2}$.

We now come to the most costly step.
By Theorem~\ref{computingtheta}, it takes time
$\smash{\widetilde{O}(P_{\mathrm{theta}}(j)^2)}$ to do
a single iteration of step~\ref{itm:evaluatetheta}.
In particular, all iterations of this step together take time 
$\smash{\widetilde{O}(\Delta_1^{\smash{7/2}}\Delta_0^{\smash{11/2}})}$.

The $j$-th iteration of
step~\ref{itm:igusainvariants}
takes time
$\widetilde{O}(P_{\mathrm{theta}}(j))$ and hence all iterations of this step
together take
time $\otilde{2}{3}$.
Finally, by Theorem~\ref{thm:polynomialmultiplication},
step~\ref{itm:dopolyfromroots} takes time $\widetilde{O}(h')$
times $\widetilde{O}(P_{\mathrm{Igusa}}(j))$, which is $\otilde{2}{3}$.
The same amount of time is needed for the final two steps.

The output consists of $h'+1$ rational coefficients,
each of which has a bit size
of $\otilde{3/2}{5/2}$, 
hence the
size of the output is $\otilde{2}{3}$.

This proves the main theorem, except when using
the polynomials $\widehat{H}_{K,n}$ ($n=2,3$) of Section~\ref{ssec:hhat}.
With our naive method of evaluating $\widehat{H}_{K,n}$,
it takes $\widetilde{O}(h_1)$ times
as much time to reconstruct $\widehat{H}_{K,n}$ from the Igusa
invariants as it does to reconstruct~$H_{K,n}$.
This $\otilde{5/2}{7/2}$
is still dominated by the running time of the rest of the algorithm.
\end{proof}

\bibliography{bib}

\def\cprime{$'$}
\begin{thebibliography}{10}

\bibitem{bbel}
Juliana Belding, Reinier Br{\"o}ker, Andreas Enge, and Kristin Lauter.
\newblock Computing {H}ilbert class polynomials.
\newblock In {\em Algorithmic Number Theory -- ANTS-VIII (Banff, 2008)}, LNCS
  5011, pages 282--295. Springer, 2008.

\bibitem{fastmultiplication}
Daniel~J. Bernstein.
\newblock Fast multiplication and its applications.
\newblock In J.~Buhler and P.~Stevenhagen, editors, {\em Surveys in Algorithmic
  Number Theory}, volume~44 of {\em MSRI Publications}, pages 325--384.
  Cambridge, 2008.

\bibitem{birkenhake-lange}
Christina Birkenhake and Herbert Lange.
\newblock {\em Complex abelian varieties}, volume 302 of {\em Grundlehren der
  mathematischen Wissenschaften}.
\newblock Springer, second edition, 2004.

\bibitem{bolza}
Oskar Bolza.
\newblock Darstellung der rationalen ganzen {I}nvarianten der {B}in\"arform
  sechsten {G}rades durch die {N}ullwerthe der zugeh\"origen
  {$\vartheta$}-{F}unctionen.
\newblock {\em Math. Ann.}, 30(4):478--495, 1887.

\bibitem{bruinier-yang}
Jan~Hendrik Bruinier and Tonghai Yang.
\newblock C{M}-values of {H}ilbert modular functions.
\newblock {\em Invent. Math.}, 163(2):229--288, 2006.

\bibitem{buchmann-lenstra}
Johannes Buchmann and Hendrik~W. Lenstra, Jr.
\newblock Approximating rings of integers in number fields.
\newblock {\em Journal de Th{\'e}orie des Nombres de Bordeaux}, 6:221--260,
  1994.

\bibitem{cardona-quer}
Gabriel Cardona and Jordi Quer.
\newblock Field of moduli and field of definition for curves of genus 2.
\newblock In {\em Computational aspects of algebraic curves}, volume~13 of {\em
  Lecture Notes Ser. Comput.}, pages 71--83. World Scientific, 2005.

\bibitem{ckl-3adic}
Robert Carls, David Kohel, and David Lubicz.
\newblock Higher-dimensional 3-adic {CM} construction.
\newblock {\em J. Algebra}, 319(3):971--1006, 2008.

\bibitem{carls-lubicz}
Robert Carls and David Lubicz.
\newblock A {$p$}-adic quasi-quadratic time point counting algorithm.
\newblock {\em Int. Math. Res. Not. IMRN}, (4):698--735, 2009.

\bibitem{deshalit-goren}
Ehud de~Shalit and Eyal~Z. Goren.
\newblock On special values of theta functions of genus two.
\newblock {\em Ann. Inst. Fourier (Grenoble)}, 47(3):775--799, 1997.

\bibitem{dupont}
R{\'e}gis Dupont.
\newblock {\em Moyenne arithm{\'e}tico-g{\'e}om{\'e}trique, suites de
  {B}orchardt et applications}.
\newblock PhD thesis, {\'E}cole Polytechnique, 2006.
\newblock
  \url{http://www.lix.polytechnique.fr/Labo/Regis.Dupont/these_soutenance.pdf}.

\bibitem{dupontarticle}
R{\'e}gis Dupont.
\newblock Fast evaluation of modular functions using {N}ewton iterations and
  the {AGM}.
\newblock {\em Math. Comp.}, 80(275):1823--1847, 2011.

\bibitem{eisenbrand-rote}
Friedrich Eisenbrand and G{\"u}nter Rote.
\newblock Fast reduction of ternary quadratic forms.
\newblock In {\em Cryptography and lattices ({P}rovidence)}, volume 2146 of
  {\em Lecture Notes in Comput. Sci.}, pages 32--44. Springer, 2001.

\bibitem{eisentrager-lauter}
Kirsten Eisentr{\"a}ger and Kristin Lauter.
\newblock A {CRT} algorithm for constructing genus 2 curves over finite fields.
\newblock In {\em Arithmetic, Geometry and Coding Theory, AGCT-10 (Marseille,
  2005)}. Soci{\'e}t{\'e} Math{\'e}matique de France, 2011.

\bibitem{enge}
Andreas Enge.
\newblock The complexity of class polynomial computation via floating point
  approximations.
\newblock {\em Math. Comp.}, 78(266):1089--1107, 2009.

\bibitem{enge-morain}
Andreas Enge and Fran{\c{c}}ois Morain.
\newblock Fast decomposition of polynomials with known {G}alois group.
\newblock In {\em Applied algebra, algebraic algorithms and error-correcting
  codes ({T}oulouse)}, LNCS 2643, pages 254--264. Springer, 2003.

\bibitem{hehcc18}
Gerhard Frey and Tanja Lange.
\newblock Complex multiplication.
\newblock In H.~Cohen, G.~Frey, R.~Avanzi, C.~Doche, T.~Lange, K.~Nguyen, and
  F.~Vercauteren, editors, {\em Handbook of elliptic and hyperelliptic curve
  cryptography}, pages 455--473. Chapman \& Hall/CRC, 2006.

\bibitem{ghkrw-2adic}
Pierrick Gaudry, Thomas Houtmann, David Kohel, Christophe Ritzenthaler, and
  Annegret Weng.
\newblock The 2-adic {CM} method for genus 2 curves with application to
  cryptography.
\newblock In {\em Advances in Cryptology -- ASIACRYPT 2006}, LNCS 4284, pages
  114--129. Springer, 2006.

\bibitem{goren}
Eyal~Z. Goren.
\newblock On certain reduction problems concerning abelian surfaces.
\newblock {\em manuscripta mathematica}, 94(1):33--43, 1997.

\bibitem{goren-lauter}
Eyal~Z. Goren and Kristin Lauter.
\newblock Class invariants for quartic {CM} fields.
\newblock {\em Annales de l'Institut Fourier}, 57(2):457--480, 2007.

\bibitem{gorenlautertoappear}
Eyal~Z. Goren and Kristin Lauter.
\newblock Genus 2 curves with complex multiplication.
\newblock {\em Int Math Res Notices}, 2012(5):1068 -- 1142, 2012.

\bibitem{gottschling}
Erhard Gottschling.
\newblock Die {R}andfl{\"a}chen des {F}undamentalbereiches der {M}odulgruppe.
\newblock {\em Math. Annalen}, 138:103--124, 1959.

\bibitem{hecke}
Erich Hecke.
\newblock {\em Vorlesungen \"uber die {T}heorie der algebraischen {Z}ahlen}.
\newblock Chelsea Publishing Co., 1970.
\newblock Second edition of the 1923 original.

\bibitem{igusa}
Jun-Ichi Igusa.
\newblock Arithmetic variety of moduli for genus two.
\newblock {\em Annals of Mathematics}, 72(3):612--649, 1960.

\bibitem{igusa-i}
Jun-Ichi Igusa.
\newblock On {S}iegel modular forms of genus two.
\newblock {\em Amer. J. Math.}, 84(1):175--200, 1962.

\bibitem{igusa-modformsprojective}
Jun-Ichi Igusa.
\newblock Modular forms and projective invariants.
\newblock {\em Amer. J. Math.}, 89(3):817--855, 1967.

\bibitem{kirrinnis}
Peter Kirrinnis.
\newblock Partial fraction decompostion in {$\mathbb{C}(z)$} and simultaneous
  {N}ewton iteration for factorization in {$\mathbb{C}[z]$}.
\newblock {\em J. Complexity}, 14(3):378--444, 1998.

\bibitem{klingen}
Helmut Klingen.
\newblock {\em Introductory lectures on {S}iegel modular forms}, volume~20 of
  {\em Cambridge Studies in Advanced Mathematics}.
\newblock Cambridge University Press, 1990.

\bibitem{echidna}
David Kohel et~al.
\newblock {ECHIDNA} algorithms for algebra and geometry experimentation.
\newblock
  {\url{http://echidna.maths.usyd.edu.au/~kohel/dbs/complex_multiplication2.html}},
  2007.

\bibitem{lattices}
Hendrik~W. Lenstra, Jr.
\newblock Lattices.
\newblock In J.~Buhler and P.~Stevenhagen, editors, {\em Surveys in Algorithmic
  Number Theory}, volume~44 of {\em MSRI Publications}, pages 127 -- 181.
  Cambridge, 2008.

\bibitem{louboutinlower}
St{\'e}phane Louboutin.
\newblock Explicit lower bounds for residues at {$s=1$} of {D}edekind zeta
  functions and relative class numbers of {CM}-fields.
\newblock {\em Trans. Amer. Math. Soc.}, 355(8):3079--3098 (electronic), 2003.

\bibitem{mestre}
Jean-Fran{\c{c}}ois Mestre.
\newblock Construction de courbes de genre {$2$} {\`{a}} partir de leurs
  modules.
\newblock In {\em Effective methods in algebraic geometry ({C}astiglioncello,
  1990)}, volume~94 of {\em Progr. Math.}, pages 313--334. Birkh\"auser, 1991.

\bibitem{tataii}
David Mumford.
\newblock {\em Tata lectures on theta {II}}, volume~43 of {\em Progress in
  Mathematics}.
\newblock Birkh\"auser, 1984.

\bibitem{schoof}
Ren{\'e} Schoof.
\newblock Computing {A}rakelov class groups.
\newblock In J.~Buhler and P.~Stevenhagen, editors, {\em Surveys in Algorithmic
  Number Theory}, volume~44 of {\em MSRI Publications}, pages 447 -- 495.
  Cambridge, 2008.

\bibitem{shimura-taniyama}
Goro Shimura and Yutaka Taniyama.
\newblock {\em Complex multiplication of abelian varieties and its applications
  to number theory}, volume~6 of {\em Publications of the Mathematical Society
  of Japan}.
\newblock 1961.

\bibitem{spallek}
Anne-Monika Spallek.
\newblock {\em Kurven vom {G}eschlecht {$2$} und ihre {A}nwendung in
  {P}ublic-{K}ey-{K}ryptosystemen}.
\newblock PhD thesis, Institut f{\"u}r Experimentelle Mathematik,
  Universit{\"a}t GH Essen, 1994.
\newblock {\url{http://www.iem.uni-due.de/zahlentheorie/AES-KG2.pdf}}.

\bibitem{sage}
William Stein et~al.
\newblock {Sage} mathematics software 4.7.2, 2011.
\newblock \url{http://www.sagemath.org/}.

\bibitem{phdthesis}
Marco Streng.
\newblock {\em Complex multiplication of abelian surfaces}.
\newblock PhD thesis, Universiteit Leiden, 2010.
\newblock \url{http://hdl.handle.net/1887/15572}.

\bibitem{Thomae}
Carl~Johannes Thomae.
\newblock Beitrag zur {B}estimmung von {$\vartheta(0,\cdots,0)$} durch die
  {K}lassenmoduln algebraischer {F}unktionen.
\newblock {\em J. reine angew. Math.}, 71:201--222, 1870.

\bibitem{vanwamelen}
Paul van Wamelen.
\newblock Examples of genus two {CM} curves defined over the rationals.
\newblock {\em Math. Comp.}, 68(225):307--320, 1999.

\bibitem{gathen-gerhard}
Joachim von~zur Gathen and J{\"u}rgen Gerhard.
\newblock {\em Modern computer algebra}.
\newblock Cambridge, second edition, 2003.

\bibitem{washington-cycl}
Lawrence~C. Washington.
\newblock {\em Introduction to Cyclotomic Fields}.
\newblock GTM 83. Springer, 1982.

\bibitem{weber3}
Heinrich Weber.
\newblock {\em Algebraische Zahlen}, volume~3 of {\em Lehrbuch der Algebra}.
\newblock Friedrich Vieweg, 1908.

\bibitem{weng}
Annegret Weng.
\newblock Constructing hyperelliptic curves of genus {$2$} suitable for
  cryptography.
\newblock {\em Math. Comp.}, 72(241):435--458, 2002.

\bibitem{yang}
Tonghai Yang.
\newblock Arithmetic intersection on a {H}ilbert modular surface and the
  {F}altings height.
\newblock \url{http://www.math.wisc.edu/~thyang/RecentPreprint.html}, 2007.

\end{thebibliography}

\end{document}